\documentclass[largeformat]{interact}

\newcommand{\hH}{\hat{\mathcal H}}

\newcommand{\tM}{\tilde M}

\newcommand{\tP}{\tilde P}

\newcommand{\tS}{\tilde{\mathcal S}}

\newcommand{\tY}{\tilde Y}
\newcommand{\tZ}{\tilde Z}

\newcommand{\ty}{\tilde y}

\newcommand{\tcA}{\tilde{\mathcal A}}
\newcommand{\tcB}{\tilde{\mathcal B}}
\newcommand{\tcC}{\tilde{\mathcal C}}
\newcommand{\tcD}{\tilde{\mathcal D}}

\newcommand{\bR}{{\mathbb R}}
\newcommand{\bZ}{{\mathbb Z}}

\newcommand{\cA}{{\mathcal A}}
\newcommand{\cB}{{\mathcal B}}
\newcommand{\cC}{{\mathcal C}}
\newcommand{\cD}{{\mathcal D}}

\newcommand{\cH}{{\mathcal H}}

\newcommand{\cJ}{{\mathcal J}}

\newcommand{\cM}{{\mathcal M}}

\newcommand{\cS}{{\mathcal S}}
\newcommand{\cT}{{\mathcal T}}

\newcommand{\cW}{{\mathcal W}}

\newcommand{\cY}{{\mathcal Y}}
\newcommand{\cZ}{{\mathcal Z}}

\newcommand{\bn}{\binom}

\newcommand{\hX}{{\hat X}}

\newcommand{\hY}{{\hat Y}}
\newcommand{\hZ}{{\hat Z}}

\newcommand{\hS}{\hat{\mathcal S}}

\newcommand{\dagg}{\dagger}
\newcommand{\deft}{\operatorname{\text{\rm def}}}
\newcommand{\diag}{\operatorname{\text{\rm diag}}}
\newcommand{\rank}{\operatorname{\text{\rm rank}}}
\newcommand{\Ker}{\operatorname{\text{\rm Ker}}}
\newcommand{\Imm}{\operatorname{\text{\rm Im}}}

\newcommand{\ind}{\operatorname{\text{\rm ind}}}

\newcommand{\const}{\operatorname{\text{\rm const}}}

\newcommand{\K}{_{k+1}}

\newcommand{\smat}[4]{\left
(\begin{smallmatrix}#1&#2\\#3&#4\end{smallmatrix}\right)}
\newcommand{\mmatrix}[1]{\left(\begin{matrix} #1
  \end{matrix}\right)}

\newcommand{\Sp}{Sp(2n)}

\newcommand{\bnn}{(0\,\,I)^T}

\newcommand{\Rbb}{{\mathbb R}}
\newcommand{\Zbb}{{\mathbb Z}}

\newcommand{\la}{\lambda}

\newcommand{\sZ}{{\scriptscriptstyle\Zbb}}

\newcommand{\tB}{{\tilde {\mathcal B}}}

\usepackage{epstopdf}
\usepackage[caption=false]{subfig}
\usepackage[numbers,sort&compress]{natbib}
\usepackage{amssymb}
\numberwithin{equation}{section}
\bibpunct[, ]{[}{]}{,}{n}{,}{,}
\renewcommand\bibfont{\fontsize{10}{12}\selectfont}

\theoremstyle{plain}
\newtheorem{theorem}{Theorem}[section]
\newtheorem{lemma}[theorem]{Lemma}
\newtheorem{corollary}[theorem]{Corollary}
\newtheorem{proposition}[theorem]{Proposition}
\newtheorem{remark}[theorem]{Remark}
\theoremstyle{definition}
\newtheorem{definition}[theorem]{Definition}
\newtheorem{example}[theorem]{Example}

\theoremstyle{remark}

\begin{document}
\title{Renormalized Oscillation Theory for Symplectic Eigenvalue Problems with Nonlinear Dependence on the Spectral Parameter}

\author{
\name{Julia Elyseeva\thanks{CONTACT Julia Elyseeva. Email: elyseeva@gmail.com}}
\affil{Department of Applied Mathematics, Moscow State University of Technology, Vadkovskii per. 3a, 101472, Moscow, Russia}
}

\maketitle

\begin{abstract}
In this paper we establish new renormalized oscillation theorems
for discrete symplectic eigenvalue problems with Dirichlet boundary
conditions. These theorems present the number of finite eigenvalues
of the problem in arbitrary interval $(a,b]$ using number of focal
points of a transformed conjoined basis associated with Wronskian
of two principal solutions of the symplectic system evaluated at
the endpoints $a$ and $b.$  We suppose that   the symplectic
coefficient matrix of the system depends nonlinearly on the
spectral parameter and that it satisfies certain natural
monotonicity assumptions. In our treatment we admit possible
oscillations in the coefficients of the symplectic system by
incorporating their nonconstant rank with respect to the spectral
parameter.
\end{abstract}

\begin{amscode}39A12; 39A21 \end{amscode}
\section{Introduction}\label{sec1}
In this paper we consider the discrete  symplectic system
\begin{equation*} \label{Sla}
  y_{k+1}(\la)=\cS_k(\la)\,y_k(\la), \quad k\in[0,N]_\sZ, \quad \la\in\Rbb,\tag{S$_\la$}
\end{equation*}
with  the Dirichlet boundary conditions
\begin{equation*} \label{E0}
   x_0(\la)=0=x_{N+1}(\la), \tag{E$_0$}
   \end{equation*}
where we use the notation $[M,N]_\sZ:=[M,N]\cap\Zbb$ for the discrete interval with endpoints
$M,N\in\Zbb$. The coefficient matrix $\cS_k(\la)\in \bR^{2n\times 2n}$ of system \eqref{Sla} with $n\times n$ blocks $\cA(\la)$, $\cB(\la)$,
$\cC(\la)$, $\cD(\la)$ depending nonlinearly on the spectral parameter $\la\in\Rbb$ is assumed to be symplectic, i.e., for all $k\in[0,N]_\sZ$ and $\la\in\Rbb$ we have
\begin{equation} \label{dnlaosc:E:Intro.4}
  \cS_k^T(\la)\,\cJ\cS_k(\la)=\cJ, \quad \cS_k(\la):=\mmatrix{\cA_k(\la) & \cB_k(\la) \\ \cC_k(\la) & \cD_k(\la)}, \quad \cJ:=\mmatrix{0 & I \\ -I & 0}.
\end{equation}
In addition, we assume that the matrix
$\cS_k(\la)$ piecewise continuously differentiable in $\la\in\Rbb$, i.e., it is continuous on $\Rbb$ and the derivative
$\dot{\cS}_k(\la):=\frac{d}{d\la}\cS_k(\la)$ is piecewise continuous in the parameter $\la\in\Rbb$ for all $k\in[0,N]_\sZ$. Given the above
symplectic matrix $\cS_k(\la)$ we consider the  monotonicity assumption
\begin{equation} \label{Smon}
  \Psi(\cS_k(\la))=\Psi_k(\la):=\cJ^T\!{\dot{\cS}}_k(\la)\,\cS_k^{-1}(\la)\geq0, \quad k\in[0,N]_\sZ, \quad
  \la\in\Rbb
\end{equation}
for the $2n\times2n$ matrix $\Psi_k(\la)$, which is symmetric for any $k\in[0,N]_\sZ$ and $\la\in\Rbb$ according to \cite{rSH12}. The notation
$A\ge 0$ means that the matrix $A$ is symmetric and nonnegative definite. Symplectic difference systems \eqref{Sla} cover as special cases many important
difference equations, such as the second order (or even order) Sturm--Liouville difference equations, symmetric three-term recurrence equations,
and linear Hamiltonian difference systems, see \cite{cdA.acP96,mB.oD97,mB98b,rSH12,wK.rSH13}. A complete review of the history  and  development of
 qualitative theory of \eqref{Sla} is given in the new monograph \cite{bookDEH} (see also the references therein).

Classical oscillation theorems connect the oscillation and spectral
theories of \eqref{Sla}. Assume that we need to know how many
eigenvalues of \eqref{Sla}, \eqref{E0} are located in the given
interval $(a,b]\subseteq \bR.$ Then, according to the global
oscillations theorems in \cite{B-D-K-RMJM,D-K-JDEA-2007} the
difference $ l_d(Y^{[0]}(b),0,N+1)-l_d(Y^{[0]}(a),0,N+1)$
of the numbers
of focal points calculated for the \textit{principal} solutions
$Y^{[0]}(b),Y^{[0]}(a)$ of \eqref{Sla} evaluated at the endpoints
$\la=a$ and $\la=b$ presents the number of eigenvalues
$\#\{\nu\in\sigma |a <\nu\le b\}$ of \eqref{Sla}, \eqref{E0} in
$(a,b]$. This result was proved in \cite{B-D-K-RMJM,D-K-JDEA-2007}
for the coefficient matrix  $\cS_k(\la)$ and $\Psi_k(\la)$ in the
form
\begin{equation} \label{dnlaosc:E:Intro.7A}
  \cS_k(\la)=\mmatrix{I & 0 \\ -\la\,\cW_k& I}S_k, \;
  \Psi_k(\la)\equiv\Psi_k:=\mmatrix{\cW_k & 0 \\ 0 & 0}\geq0, \; k\in[0,N]_\sZ, \quad \la\in\Rbb,
\end{equation}
 where $S_k$ is a constant symplectic matrix. The same result for \eqref{Sla} with  the general nonlinear
dependence on $\la$ was originally proved in \cite{rSH12} for
$\cB_{k}(\la)=\const$ (here $\cB_k(\la)$ is the block of $\cS_k(\la)$ given by \eqref{dnlaosc:E:Intro.4}) and then generalized in \cite{wK.rSH13} to the case
\begin{equation}\label{rBconst}
  \rank \cB_k(\la)=\const,\quad \la\in\bR,\,k\in[0,N]_{\sZ}.
\end{equation}
Then it was shown in \cite{jE15} that assumption \eqref{rBconst} plays a crucial role in the oscillation theory, in particular, if \eqref{rBconst} is violated the number of focal points of the principal solution of \eqref{Sla} loses the monotonicity with respect to $\la$ and then the difference $l_d(Y^{[0]}(b),0,N+1)-l_d(Y^{[0]}(a),0,N+1)$ can be negative. In this case it is necessary to incorporate oscillations of the block $\cB_k(\la)$ to present a proper generalization of the results in \cite{B-D-K-RMJM,D-K-JDEA-2007,rSH12,wK.rSH13}.
Moreover, it was
proven in \cite[Corollary~2.5]{jE15} that condition
\begin{equation}\label{finrank}
  \rank\cB_{k}(\lambda)=\rank\cB_{k}(\lambda^{-}) \quad\text{for all } \la\leq\la_0,\quad k\in[0,N]_\sZ,
\end{equation}
holds for some $\lambda_{0}\in\bR$ if and only if the real spectrum of \eqref{Sla},\eqref{E0} is bounded from below.
 Observe that for a~fixed $k\in[0,N]_\sZ$ the symplectic matrix $\cS_k(\la)$ can be viewed as the fundamental matrix of the linear Hamiltonian
differential system (with respect to $\la$)
\begin{equation} \label{PsiHam}
  \dot{\cS}_k(\la)=\cJ\,\Psi_k(\la)\,\cS_k(\la), \quad \la\in\Rbb,
\end{equation}
with the symmetric Hamiltonian $\Psi_k(\la)\ge 0$ given by
\eqref{Smon}. In this context one can introduce the numbers
\begin{equation}\label{jump}
 \vartheta(\cS_k(\lambda_0))= \vartheta_k(\lambda_0) := \rank \cB_{k}(\lambda_0^{-})-\rank\cB_{k}(\lambda_0),
 \quad k\in[0,N]_\sZ
\end{equation}
describing the multiplicities  of \textit{proper focal points} (see \cite{Kratz-Analysis}) of $\cS_k(\la)\bnn$ as a conjoined basis of \eqref{PsiHam} and then  condition \eqref{finrank} means that
system \eqref{PsiHam} is \textit{nonoscillatory} for $\la$ near
$-\infty$. In the recent paper \cite{jvE.rSH18} we generalized  the results in \cite{jE15} to the case of symplectic eigenvalue problems with general self-adjoint boundary conditions admitting possible oscillations of their coefficients with respect to $\la\in\bR.$

The \textit{renormalized} and the more general \textit{relative}
oscillation theory makes it possible to replace the difference $l_d(\hY(b),0,N+1)-l_d(Y(a),0,N+1)$  of the numbers of focal points calculated for
$\la = a$ and $\la = b$ by the number of focal points of only one transformed conjoined basis
$\tY_k(a,b)$ associated with
the Wronskian $\hY_k^T(b)\cJ Y_k(a)$ of $Y_k(a)$ and $\hY_k(b).$ Remark that we refer to the renormalized oscillation theory of
\eqref{Sla} when the consideration concerns oscillations of the
Wronskian of two conjoined bases of \eqref{Sla} considered for
different values of $\la.$ The \textit{relative} oscillation theory
investigates the oscillatory behavior of Wronskians of conjoined
bases of two symplectic systems with different coefficient matrices
$\cS_k(\la)$ and $\hS_k(\la),$ then all results of the renormalized
theory follow from the relative oscillation theorems for the case
$\cS_k(\la)=\hS_k(\la),\,\la\in\bR,\,k\in[0,N]_{\sZ}.$

The  {relative} oscillation theory was
developed for eigenvalue problems for the second order
Sturm--Liouville difference and differential  equations (with
linear dependence on $\la$) in
\cite{Gesztesy1,Kruger,teshl3,teschl1,amman}.  In the recent papers
\cite{Gesztesy2,Howard} the renormalized oscillation theory in
\cite{Gesztesy1} is extended  to the case of general linear
Hamiltonian systems with block matrix coefficients, which are
continuous counterparts of \eqref{Sla}.

The relative oscillation theory for two symplectic problems with Dirichlet boundary conditions under restriction \eqref{dnlaosc:E:Intro.7A} is presented in \cite{E-DE-2010,E-AML2010}, in \cite[Theorem 5]{E-AML2012}  the renormalized oscillation theorem  for  \eqref{Sla},  \eqref{dnlaosc:E:Intro.7A}   is
 extended to the case of general self-adjoint boundary conditions. For the case of general nonlinear dependence on $\la$ the first results of the relative oscillation theory for two matrix Sturm-Liouville equations were proved in \cite{E-ADV}. In \cite[Section 6.1]{bookDEH} we presented the relative oscillation theory for two symplectic eigenvalue problems with nonlinear dependence on $\la$ and with general self-adjoint boundary conditions. All these results are derived under restriction \eqref{rBconst} for $\cS_k(\la).$

The main results of this paper are devoted to the renormalized oscillation theory for \eqref{Sla},\eqref{E0} without condition \eqref{rBconst}. In this situation the classical oscillation theorem (see \cite{jE16}) presents the number of finite eigenvalues $\#\{\nu\in\sigma |a <\nu\le b\}$ of problem \eqref{Sla}, \eqref{E0} in $(a,b]$ incorporating oscillations of $\cB_k(\la)$ in terms of  numbers \eqref{jump}   (see  Theorem~\ref{globosc} in Section~\ref{sec2})
 \begin{equation}\label{eig.in.int2}
    l_d(Y^{[0]}(b),0,N+1)-l_d(Y^{[0]}(a),0,N+1)+\sum\limits_{a<\nu\le b}\sum\limits_{k=0}^{N}\vartheta_{k}(\nu)=\#\{\nu\in\sigma |a <\nu\le b\}.
  \end{equation}
A similar formula can be proved for the so-called \textit{backward focal points} $l_d^*(Y^{[N+1]}(\la),0,N+1)$ of the principal solutions at $N+1$ (see Theorem~\ref{globosc*} in Section~\ref{sec2}). The main results of the paper  (see Theorems~\ref{globoscrel1},~\ref{globoscrel1*}) present renormalized versions of  Theorems~\ref{globosc},~\ref{globosc*}, respectively. In more details, introducing a fundamental matrix $Z_k^{[N+1]}(\la)$ of \eqref{Sla}   with the initial condition $Z_{N+1}^{[N+1]}(\la)=I$   we have instead of \eqref{eig.in.int2} the following renormalized formula
\begin{equation}\label{renormeiga1}\begin{aligned}
    l_d((Z^{[N+1]}(a))^{-1}Y^{[0]}(b),0,N+1)+\sum\limits_{a<\nu\le b}\sum\limits_{k=0}^{N}\tilde{\vartheta}_{k}(\nu)=\#\{\nu\in\sigma|\,a<\nu\le b\},
 \end{aligned} \end{equation}
where the numbers $\tilde{\vartheta}_{k}(\nu)$ are associated with the transformed coefficient matrix $\tS_k(\la)=(Z_{k+1}^{[N+1]}(a))^{-1}\cS_k(\la)Z_k^{[N+1]}(a)$ by analogy with \eqref{jump}.
In \eqref{renormeiga1} we have the number $l_d((Z^{[N+1]}(a))^{-1}Y^{[0]}(b),0,N+1)$ which describes oscillations of the transformed conjoined basis $(Z_k^{[N+1]}(a))^{-1}Y_k^{[0]}(b)$ associated with the Wronskian $Y_k^{[N+1]\;T}(a)\cJ Y_k^{[0]}(b)$ of the principal solutions $Y^{[N+1]}_k(a),\,Y_k^{[0]}(b)$ of \eqref{Sla}.  The major advantage of using  \eqref{renormeiga1} instead of \eqref{eig.in.int2}  is   the calculation of  only one number $l_d((Z^{[N+1]}(a))^{-1}Y^{[0]}(b),0,N+1)$ instead of
$l_d(Y^{[0]}(b),0,N+1),\,l_d(Y^{[0]}(a),0,N+1)$ especially in case of  highly oscillatory principal solutions $Y^{[0]}(b),\,Y^{[0]}(a)$. The price of this advantage is  the necessity to evaluate the second addend in \eqref{renormeiga1} which depends on the fundamental matrix $Z^{[N+1]}(\la)$ of system \eqref{Sla} evaluated for $\la=a.$
In Section~\ref{sec4} of the paper we decide this problem presenting \eqref{renormeiga1} in an invariant form incorporating oscillations of $\cS_k(a)-\cS_k(\la),\,\la\in(a,b]$ instead of oscillations of blocks of $\tS_k(\la).$ We proved in Section~\ref{sec4} (see Theorem~\ref{globoscbig}) that \eqref{renormeiga1} is equivalent to
 \begin{equation}\label{renormeiga2}\begin{array}{c}
    L_d((Z^{[N+1]}(a))^{-1}Y^{[0]}(b),0,N+1)+\sum\limits_{a<\nu\le b}\sum\limits_{k=0}^{N}{\rho}_{k}(\nu)=\#\{\nu\in\sigma|\,a<\nu\le b\},\\[4mm]
    \rho_{k}(\la)=\rank(\cS_k(a)-\cS_k(\la^-))-\rank(\cS_k(a)-\cS_k(\la))\ge 0,
 \end{array} \end{equation}
where $L_d((Z^{[N+1]}(a))^{-1}Y^{[0]}(b),0,N+1)$ is the number of (forward) focal points of a $4n\times 2n$ conjoined basis associated with $Z^{[N+1]}(a))^{-1}Y^{[0]}(b)$ (see Remark~\ref{aboutcon}). We call representation \eqref{renormeiga2} \textit{invariant} because after the replacement of the matrices $\cS_k(\la),\,k\in [0,N]_{\sZ}$ by $R_{k+1}^{-1}\cS_k(\la) R_k$  formula \eqref{renormeiga2} stays the same. Here $R_k,\,k\in [0,N+1]_{\sZ}$ is an arbitrary sequence of symplectic transformation matrices which do not depend on $\la.$ In the last part of Section~\ref{sec4} we investigate the renormalized oscillation theory for systems \eqref{Sla} under the assumption $\rho_{k}(\la)=0,\,k\in [0,N]_{\sZ},\,\la\in[a,b]$  which is necessary and sufficient for the equality $L_d((Z^{[N+1]}(a))^{-1}Y^{[0]}(b),0,N+1)=\#\{\nu\in\sigma|\,a<\nu\le b\}.$ In particular, we show that this equality holds for any Hamiltonian difference system (see Corollary~\ref{renham}) under the monotonicity assumption $\dot{\cH}_k(\la)\ge 0,\,k\in [0,N]_{\sZ},\,\la\in[a,b]$ for the discrete Hamiltonian  ${\cH}_k(\la).$

For the proof of the renormalized theorems  we  involve  new results of the oscillation theory for continuous case -- for the differential Hamiltonian systems in form \eqref{PsiHam}. Indeed the theory presented in the paper is now a combination of two oscillation theories -- for the discrete and for the continuous case.  Using comparison theorems for the differential case (see \cite[Theorem 2.2]{jE16}) we present a new interpretation of the results of the discrete spectral theory in \cite{B-D-K-RMJM,D-K-JDEA-2007,rSH12,wK.rSH13,jE15,jvE.rSH18} which helps to provide the proof of Theorems~\ref{globoscrel1},~\ref{globoscrel1*} in a compact form.  In more details, consider a  symplectic fundamental matrix $Z_k(\la)$ of \eqref{Sla} associated with the conjoined basis $Y_k(\la)=Z_k(\la)\bnn$  under the monotonicity assumption $\Psi(Z_k(\la))\ge 0$ with respect to $\la.$ Then for arbitrary sequence of symplectic matrices $R_k,\,k\in [0,N+1]_{\sZ}$  the matrix $R_k^{-1}Y_k(\la)$  can be considered as a function of $\la\in[a,b]$ with the (continuous) number of proper focal points $l_c(R_k^{-1}Y_k,a,b)$ for any fixed index $k$ and similarly, as a function of the discrete variable $k$ with the (discrete) number of focal points  $l_d(R^{-1}Y(\la_0),0,N+1)$ for any fixed $\la=\la_0.$ Introduce the following closed path $\la=a,\,k\in[0,N+1]_\sZ$;  $\la\in[a,b],\,k=N+1;$ $\la=b,\,k\in[0,N+1]_\sZ;$ $\la\in[a,b],\,k=0$  in the plane $(\la,k).$ Then, according to Theorem~\ref{forward} proved in Section~\ref{sec3} we have the following representation for the sum of all focal points (in the continuous and in the discrete settings) along this path
\begin{equation}\label{feb5}\begin{aligned}
  l_d(R^{-1}Y(a),0,N+1)+l_c(R_{N+1}^{-1}Y_{N+1},a,b)&-l_d(R^{-1}Y(b),0,N+1)-l_c(R_0^{-1}Y_0,a,b)\\&=\sum\limits_{k=0}^N l_c(\tS_k\bnn,a,b),
\end{aligned}\end{equation}
where $l_c(\tS_k\bnn,a,b)$ is the number of focal points of $\tS_k(\la)\bnn$ for $\la\in(a,b]$ and $\tS_k(\la)=R_{k+1}^{-1}\cS_k(\la)R_k.$ For the case $R_k:=I,\,k\in [0,N]_{\sZ}$ and $Z_0(\la):=I$ formula \eqref{feb5} turns into \eqref{eig.in.int2} because  $Y_k(\la):=Y^{[0]}_k(\la),$  $l_c(Y_0,a,b)=0,$ and the quantity $l_c(Y_{N+1},a,b)$ presents the number of finite eigenvalues of \eqref{Sla}, \eqref{E0}. Observe that $l_c(Y_{N+1},a,b)$ stays the same for the case of the nonconstant matrices $R_k$ under the condition $R_{N+1}=I.$ Then one can derive from \eqref{feb5} the representation of the number $\#\{\nu\in\sigma|\,a<\nu\le b\}$  using the transformed principal solutions $R_k^{-1}Y_k^{[0]}(\la)$ (see Theorem~\ref{newmay2019}).

The renormalized oscillation theory is connected with further simplifications of \eqref{feb5} when  $R_k$ is chosen in such a way that $R_k^{-1}Y_k^{[0]}(\la)$ does not depend on $k=0,1,\dots,N+1$ for $\la=b$ or for $\la=a.$ In particular, formula \eqref{renormeiga1} is derived  for the case $R_k:=Z_k^{[N+1]}(a)$, when $l_d(R^{-1}Y^{[0]}(a),0,N+1)=~0.$ In the proof of the equivalence of \eqref{renormeiga1} and \eqref{renormeiga2} we also use separation results for the differential case (see \cite[Theorem 2.3]{jE16} and \cite[Theorem~4.1]{pS.rSH17}) to  connect (in an explicit form) the numbers $\rho_{k}(\la)$ in \eqref{renormeiga2} with $\tilde{\vartheta}_{k}(\nu)$ and ${\vartheta}_{k}(\nu)$ in \eqref{renormeiga1} and \eqref{eig.in.int2} (see Lemma~\ref{newosc1}).

The paper is organized as follows. In the next section we recall main notions of the discrete spectral and oscillation theory. We  recall the classical global oscillation theorem from \cite{jE15} (see Theorem~\ref{globosc}) and present the version of this theorem in terms of backward focal points of the principal solution at $ N + 1$  (see Theorem~\ref{globosc*}). In Section~\ref{sec3} we recall some basic notions of the oscillation theory of linear differential Hamiltonian systems \cite{Kratz-Analysis,Hilsch4,BHK} including the recent results from \cite{jE16,pS.rSH17} and prove Theorem~\ref{globoscrel1} and Theorem~\ref{globoscrel1*} which are the renormalized versions of Theorems~\ref{globosc},~\ref{globosc*}. In Section~\ref{sec4} we derive Theorem~\ref{globoscbig}  which generalizes  \cite[Theorem 6.4]{bookDEH} to the case when assumption \eqref{rBconst} is omitted. As an application of the results of Section~\ref{sec4}   we present the renormalized oscillation theorem for the discrete linear Hamiltonian systems (see Example~\ref{Hamex}).
\section{Classical oscillation and spectral theory for symplectic systems with the Dirichlet boundary conditions}\label{sec2}
We will use the following notation. For a matrix $A,$ we denote by
$A^T,\,A^{-1},$  $A^{\dag}, \rank A,$ $\,\Ker A,\,\Imm A,$
$ \ind A,$ $ A \ge~0, A \le 0,$ respectively, its transpose, inverse,
 Moore-Penrose pseudoinverse,  rank (i.e.,
the dimension of its image), kernel, image, index (i.e., the number
of its negative eigenvalues), positive semidefiniteness, negative
semidefiniteness. If $A(t)$ is a matrix-valued function, then by
$\rank A(t^{-}_{0}),\,\ind A(t^{-}_{0})$ ($\rank A(t^{+}_{0}),$ $
\ind A(t^{+}_{0})$) we mean the left-hand (the right-hand) limits
of these quantities at $t_{0}$, provided these limits exist. We use
the notation $A(\la)|_a^b$ for the substitution $A(b)-A(a)$ for
functions of continuous argument $\la,$ and the similar notation
$A_k|_M^N=A_N-A_M$ for functions of discrete argument $k.$ We
denote by $\Delta A_k=A_{k+1}-A_k$  the forward difference
operator. We say that $A(\la)\in C_p^1$ if  $A(\la)$ is continuous with  a piecewise continuous derivative (with respect to $\la$), and use the notation $\Sp$ for the real matrix symplectic group in dimension $2n$, i.e., for real matrices $\cS$ with the condition $\cS^T\cJ\cS=\cJ.$

In this section we recall some important notions and results of the discrete oscillation and spectral theory (see \cite[Chapters~4,5]{bookDEH}).

Recall that $2n\times n$ matrix solution
$Y(\la)=\binom{X(\la)}{U(\la)}$ of \eqref{Sla} with $n\times n$ matrices $X(\la),U(\la)$ is
said to be a {\it conjoined basis} if
\begin{equation}                            \label{conjoined}
\rank\,\bn{X_k(\la)}{U_k(\la)}=n \quad\text{and}\quad X_k(\la)^TU_k(\la)=U_k(\la)^TX_k(\la).
\end{equation}

Recall the definition of forward focal points and their multiplicities for conjoined bases of \eqref{Sla}, see \cite[Definition~1]{K-JDEA}. We say that a~conjoined basis $Y(\la)$ of \eqref{Sla} has a~{\em forward focal point\/} in the
real interval $(k,k+1]$ provided
  $m_d (Y_k(\la)):=\rank M_k(\la)+\ind P_k(\la)\geq1$
and then the number $m_d (Y_k(\la))$ is its {\em multiplicity}, where
\begin{equation*} \label{dnlaosc:E:MTP.def}
   \begin{array}{c}
    M_k(\la):= (I-X_{k+1}(\la)\,X_{k+1}^\dagger(\la))\,\cB_k(\la), \quad
    T_k(\la) :=  I-M_k^\dagger(\la)\,M_k(\la), \\[1mm]
    P_k(\la) := T_k(\la)\,X_k(\la)\,X_{k+1}^\dagger(\la)\,\cB_k(\la)\,T_k(\la),
  \end{array}
\end{equation*}
and the matrix $P_k(\la)$ is symmetric. By a similar way (see \cite{D-Kyoto,E-DE-2009}) one can introduce  the multiplicities of backward focal points of a conjoined basis of \eqref{Sla} in the
real interval $[k,k+1)$. We define
 $ m_d^* (Y_k(\la)):=\rank \tM_k(\la)+\ind \tP_k(\la),$
where
\begin{gather*}                                 
    \tilde M_{k}(\la):=(I-X_{k}(\la)X_{k}^{\dagg}(\la))\,\cB^{T}_{k}(\la),\quad
      \tilde T_k(\la):=I-\tilde M_k(\la)^\dagg \tilde M_k(\la), \label{bM-definition} \\
    \tilde P_k(\la):=\tilde T_{k}^{T}(\la) X_{k+1}(\la) X_{k}^{\dagg}(\la)\cB^{T}_{k}(\la)\tilde T_{k}(\la). \label{bP-definition}
  \end{gather*}
For the numbers $ m_d(Y_k(\la)),\, m_d^* (Y_k(\la))$ we have the estimates
\begin{equation}\label{eatnum}
 0 \le m_d(Y_k(\la))\le \rank \cB_k(\la)\le n,\quad 0 \le m_d^*(Y_k(\la))\le \rank \cB_k(\la)\le n.
\end{equation}
Introduce the notation
\begin{equation} \label{E:n1.def}
  l_d(Y(\la),0,N+1):=\sum_{k=0}^{N} m_d(Y_k(\la)),\quad l_d^*(Y(\la),0,N+1):=\sum_{k=0}^{N} m_d^*(Y_k(\la))
\end{equation}
for the number of forward (resp. backward) focal points of the
conjoined basis $Y(\lambda)$ in the intervals $(0,N+1]$ (resp.
$[0,N+1)$) including their multiplicities.

Recall also that the conjoined basis $Y^{[l]}_k(\la)$ with the initial condition $Y^{[l]}_l(\la) =\bnn$ for $k=l$ is called the principal solution of \eqref{Sla} at $l$ (for $l=0,\dots,N+1$).
The following important connection  for the number of focal points of the principal solutions $Y^{[0]}(\la),\,Y^{[N+1]}(\la)$ is proven in \cite[Lemma 3.3]{E-DE-2009}
\begin{equation}\label{0-N+1}
  l_d(Y^{[0]}(\la),0,N+1)=l^*_d(Y^{[N+1]}(\la),0,N+1).
\end{equation}

To formulate the main results of this paper we recall the notion of the comparative index introduced and collaborated   in \cite{E-DE-2009,E-DE-2010} (see also \cite[Chapter 3]{bookDEH}). According to
  \cite{E-DE-2009}, for two real $2n\times n$ matrices $Y$ and $\hat Y$ satisfying condition \eqref{conjoined} we define
  their {\it comparative index\/} $\mu(Y,\hat Y)$ and the {\it dual comparative index\/} $\mu^*(Y,\hat Y)$ by
  \begin{equation*}
    \mu(Y,\hat Y):=\rank\cM+\ind\cD, \quad \mu^*(Y,\hat Y):=\rank\cM+\ind(-\cD),
  \end{equation*}
  where the $n\times n$ matrices $\cM$ and $\cD$ are defined for $X:=(I\, 0)\,Y, \quad \hat X:=(I\, 0)\,\hat Y$ as
  \begin{equation*}
     \begin{array}{c}
      \cM:=(I-XX^{\dag})\,\hat X, \quad \cD:=\cT\,w^T(Y,\hat Y)\,X^{\dag}\hat X\,\cT,
       \quad \cT:=I-\cM^{\dag}\cM,
    \end{array}
  \end{equation*}
  and where $w(Y,\hat Y)$ is the Wronskian of $Y$ and $\hat Y$ given by $ w(Y,\hY)=Y^T\cJ\hY.$

 Note that the dual index can be presented in the form
\begin{equation}\label{dualrep}
 \mu^*(Y,\hY)=\mu(P_3Y,P_3\hY),\,P_3=\diag\{-I,I\},
\end{equation}
where we use the notation from the book \cite{bookDEH}.
Moreover, one can verify that for arbitrary symplectic matrix $W$
we have that $P_{3}WP_{3}$ is symplectic as well. The matrix
$P_{3}$ together with $P_{1}=\mmatrix{0&I\\I&0}$ and $P_{2}=-P_{3}$
 plays an important role in the duality principle in the comparative
index theory  for \eqref{Sla} (see  \cite[Chapter 3]{bookDEH}). We have the following estimate for $\mu(Y,\hY)$ and $\mu^*(Y,\hY)$
\begin{equation}
 \max\{ \mu(Y, \hat Y),\mu^*(Y, \hat Y)\}\le \min\,\{\rank \hX,\,\rank w(Y,\hat Y)\}\le n. \label{prop4}
\end{equation}

We also need the following representation of $m_d (Y_k(\la))$ and
$m_d^* (Y_k(\la))$ in terms of the comparative index (see
\cite[Lemmas~3.1,3.2]{E-DE-2009}). Let $Y_k(\la)$ be a conjoined basis of \eqref{Sla} associated with a symplectic fundamental matrix $Z_k(\la)$ of this system, such that $Y_k(\la)=Z_k(\la)\bnn,$ then the multiplicities of forward and backward focal points are given by
\begin{align}\label{forwind}
  m_d (Y_k(\la))=\mu(Y_{k+1}(\la),\cS_k(\la)\bnn)&=\mu^*(Z_{k+1}^{-1}(\la)\bnn,Z_{k}^{-1}(\la)\bnn),\\
\label{backwind}
  m_d^* (Y_k(\la))=\mu^*(Y_{k}(\la),\cS_k^{-1}(\la)\bnn)&=\mu(Z_{k}^{-1}(\la)\bnn,Z_{k+1}^{-1}(\la)\bnn).
\end{align}

Remark that  assumption \eqref{Smon} on the coefficient matrix
$\cS_k(\la)$ plays a crucial role in  the results of this paper.
In the following proposition we recall some properties of
$\Psi(\cdot)$ from \cite[Propositions 2.3,~3.1]{jvE.rSH18} and add
several new ones needed for the subsequent proofs.
 \begin{proposition} \label{psiZmod}
Assume that $\cS_k(\la)\in C_p^1,\,\la\in\Rbb,\,
  k\in[0,N]_\sZ$ and \eqref{dnlaosc:E:Intro.4}, \eqref{Smon} hold. Then the following assertions are true.
 \begin{enumerate}
\item[(i)] Condition \eqref{Smon} is equivalent to
    $\Psi(\cS_k^{-1}(\la))\le 0$ or
\begin{equation}\label{invdual}
  \Psi(P_3\cS_k^{-1}(\la)P_3)\ge 0,\quad P_3=\diag\{-I,I\},
\end{equation}
\item[(ii)] If  $\cZ_k(\lambda)$ for $k\in[0,N+1]_\sZ$ is
    a~symplectic fundamental matrix of system \eqref{Sla} such
  that $\cZ_0(\la)\in C_p^1,\,\la\in\Rbb$ and
    \begin{equation}\label{monZ} \Psi(\cZ_{k}(\lambda))\geq 0,
    \quad\lambda\in\mathbb{R},
  \end{equation}
  for $k=0$, then $\cZ_k(\la)\in C_p^1,\,\la\in\Rbb$ for all $k\in[0,N+1]_\sZ$ and
  condition \eqref{monZ} holds for every $k\in[0,N+1]_\sZ$.
\item[(iii)] If instead of \eqref{monZ} we have for $k=N+1$
    \begin{equation}\label{monZ*}
    \Psi(\cZ_{k}(\lambda))\leq 0, \quad\lambda\in\mathbb{R},
  \end{equation}
where a symplectic fundamental matrix $\cZ_k(\la)$ such that $\cZ_{N+1}(\la)\in C_p^1,\,\la\in\Rbb$, then condition \eqref{monZ*} holds for all
  $k\in[0,N+1]_\sZ.$ Moreover,  \eqref{monZ*} is equivalent
  to
  \begin{equation}\label{monZP*}
    \Psi(P_3\cZ_{k}(\lambda)P_3)\geq 0,\quad \lambda\in\mathbb{R},
  \end{equation}
where $P_{3}$ is defined by \eqref{invdual}.
\item[(iv)] For arbitrary symplectic  matrices $R$ and $P$
    condition \eqref{Smon} is equivalent to
\begin{equation}\label{SmonR}
    \Psi(R^{-1}\cS_k(\lambda)\,P)\ge 0,\;\la\in\bR,\,k\in[0,N]_{\bZ}.
\end{equation}
  \end{enumerate}
\end{proposition}
\begin{proof}
 The equivalence of \eqref{Smon} and $\Psi(\cS_k^{-1}(\la))\le 0$
is proved in \cite[Proposition 2.3~(iv)]{jvE.rSH18} while
\eqref{invdual} is equivalent to $\Psi(\cS_k^{-1}(\la))\le 0$
because of the relation
$\Psi(P_3\cS_k^{-1}(\la)P_3)=-P_3\Psi(\cS_k^{-1}(\la))P_3\ge 0$ which
can be easily  verified by direct computations.

The main statement in (ii) is proved in \cite[Proposition
3.1]{jvE.rSH18}. The equivalence of \eqref{monZ*} and
\eqref{monZP*} can be proved similarly to (i).

The proof of \eqref{monZ*} for all $k\le N+1$ provided condition
\eqref{monZ*} holds for $k=N+1$ follows from the relation (see
\cite[Proposition 2.3 (i)]{jvE.rSH18})
$$\Psi(\cZ_k(\la))=\Psi(\cS_k^{-1}(\la)\cZ_{k+1}(\la))=\Psi(\cS_k^{-1}(\la))+\cS_k^T(\la)\Psi(\cZ_{k+1}(\la))\cS_k(\la)$$ by
induction using Proposition~\ref{psiZmod}(i).

The equivalence of \eqref{Smon} and \eqref{SmonR} was proved in
 \cite[Proposition 2.3(iv)]{jvE.rSH18}.
The proof is completed.
\end{proof}
 \begin{remark}\label{constInitConj}
\par(i) In particular, condition \eqref{monZ} holds  for the symplectic fundamental matrix $Z_k^{[0]}(\la)$
with the initial condition $Z_0^{[0]}(\la)=I.$ Recall that in this
case $Z_k^{[0]}(\la)\bnn=Y_k^{[0]}(\la),$ where $Y_k^{[0]}(\la)$ is
the principal solution of \eqref{Sla} at $k=0$. By a similar way we
will use condition \eqref{monZP*} for the symplectic fundamental matrix
$Z_k^{[N+1]}(\la),\,Z_{N+1}^{[N+1]}(\la)=I$ associated with the
principal solution $Y_k^{[N+1]}(\la)$ at $N+1.$
\par(ii) It was proved in \cite{jE15}, \cite[Theorem~2.4]{jvE.rSH18}, \cite[Theorem~5.1]{bookDEH} that the monotonicity assumption \eqref{Smon}  implies that for the block $\cB_k(\la)$ of
$\cS_k(\la)$ the sets $\Ker{\cB}_k(\la),\,\Imm{\cB}_k(\la)$ are
piecewise constant with respect to $\la\in \bR$. By
Proposition~\ref{psiZmod}(ii),(iii) (see also \cite{wK.rSH13}, \cite[Corollary
3.2]{jvE.rSH18}, \cite[Theorem~5.3]{bookDEH})  the same property also holds for the blocks
$X_k^{[l]}(\la)$ of the principal solutions
$Y_k^{[l]}(\la),\,l\in\{0,\,N+1\}$ of \eqref{Sla}, i.e.,
$\Ker{X}_k^{[l]}(\la),\,\Imm X_k^{[l]}(\la)$ are piecewise constant
with respect to $\la\in \bR.$ Moreover, by
Proposition~\ref{psiZmod}(iv) a similar property also holds for
blocks of the transformed matrices $R^{-1}\cS_k(\lambda)\,P$ and
$R^{-1}Z_k^{[l]}(\lambda)P,\,l \in\{0,\,N+1\}$ where $R$ and $P$
are arbitrary   symplectic matrices.
\par(iii) The renormalized oscillation theory is closely related to
the transformation theory \cite{mB.oD97},
\cite[Section~4.4]{bookDEH} of \eqref{Sla}. Instead of \eqref{Sla}
consider the transformed symplectic system
\begin{equation*} \label{tSla}\begin{array}{l}
    \ty_{k+1}(\la) =\tS_k(\la)\ty_k(\la)\quad k\in[0,N]_\sZ, \quad \tS_k(\la)=\smat{\tcA_k(\la)}{\tcB_k(\la)}{\tcC_k(\la)}{\tcD_k(\la)}\tag{TS$_\la$}
\end{array}\end{equation*}
derived from \eqref{Sla} using the transformation
$\ty_k(\la)=R^{-1}_{k}y_k(\la).$
Here $R_{k}\in \Sp$  for $k\in[0,N+1]_{\sZ}$ and does not
depend on $\la.$ Then, the coefficient matrix $\tS_k(\la)$ of system \eqref{tSla}  is symplectic
and according to Proposition~\ref{psiZmod}(iv) obeys the
monotonicity condition
\begin{equation}\label{tSmon}
    \Psi(\tS_{k}(\la))\ge 0,\quad \tS_{k}(\la)=R_{k+1}^{-1}\cS_{k}(\la)R_{k},\,k\in[0,N]_{\sZ},\,\la\in\bR.
\end{equation}
By  Proposition~\ref{psiZmod}(ii)--(iv) we also have the monotonicity
conditions for the transformed fundamental matrices
$R^{-1}_{k}Z_{k}^{[0]}(\la)$ and $R^{-1}_{k}Z_{k}^{[N+1]}(\la),$
i.e.,
\begin{equation*}\label{tZR}
\Psi(R^{-1}_{k}Z_{k}^{[0]}(\la))\ge 0,\quad \Psi(P_{3}R^{-1}_{k}Z_{k}^{[N+1]}(\la)P_{3})\ge 0,\,k\in[0,N+1]_{\sZ},\,\la\in\bR.
\end{equation*}
In particular, taking the transformation matrix $R_k$ in form
$R_k:=Z^{[N+1]}_k(\beta)$ or $R_k:=Z^{[0]}_k(\beta)$ for
$\beta\in\{a,b\},$ where $a,\,b\in\bR$ are fixed values of $\la$ we
investigate oscillations of the Wronskians
\begin{equation*} \label{wrons-trans}
   \begin{aligned}
    w(Y_k^{[N+1]}(\beta),Y_k^{[0]}(\la))=Y_k^{[N+1]\;T}(\beta)
    \cJ Y_k^{[0]}(\la)&=-(I\;0)\,(Z^{[N+1]}_k(\beta))^{-1}\,Y_k^{[0]}(\la), \\
    w(Y_k^{[0]}(\beta),Y_k^{[N+1]}(\la))=Y_k^{[0]\;T}(\beta)\cJ Y_k^{[N+1]}(\la)&=-(I\;0)\,(Z^{[0]}_k(\beta))^{-1}\,Y_k^{[N+1]}(\la)
     \end{aligned}
\end{equation*}
associated with the upper blocks of $(Z^{[N+1]}_k(\beta))^{-1}\,Y_k^{[0]}(\la)$ or $(Z^{[0]}_k(\beta))^{-1}\,Y_k^{[N+1]}(\la).$
 \end{remark}
Based on Remark~\ref{constInitConj}(i),(ii) one can  define  the
notion of a~finite eigenvalue of problem \eqref{Sla},\eqref{E0}, see
\cite[Definition~4.4]{rSH12}.
\begin{definition} \label{dnlaosc:D:finite.eigenvalue}
  Assume \eqref{dnlaosc:E:Intro.4}, \eqref{Smon} and consider the principal solution $Y_{k}^{[0]}(\la)$   of \eqref{Sla} at $k=0.$ Then, a~number $\la_0\in\Rbb$ is a~{\em finite eigenvalue\/} of problem \eqref{Sla},\eqref{E0} if
  \begin{equation*} \label{dnlaosc:E:theta.def}
    \theta(\la_0):=\rank X^{[0]}_{N+1}(\la_0^-)-\rank X^{[0]}_{N+1}(\la_0)\geq1.
  \end{equation*}
In this case the number $\theta(\la_0)$ is called the {\em algebraic multiplicity\/} of $\la_0$.
\end{definition}
It follows from Definition~\ref{dnlaosc:D:finite.eigenvalue} that
under \eqref{dnlaosc:E:Intro.4} and \eqref{Smon} the finite
eigenvalues of \eqref{Sla},\eqref{E0} are isolated.

By a similar way, Remark~\ref{constInitConj}(ii) implies that
 the quantities $\rank
\cB_k(\la)$ are constant on some left and right neighborhoods of
$\la_0$. Here $\cB_k(\la)$ is the block of $\cS_k(\la)$ given by
\eqref{dnlaosc:E:Intro.4}. Then we introduce the notation \eqref{jump} (see Section~\ref{sec1})
to describe jumps of $\rank \cB_k(\la)$ in the left neighborhood of $\la_0.$

 Using these notions we recall the global oscillation theorem for problem \eqref{Sla},\eqref{E0} for the case when
 the block $\cB_k(\la)$ of the matrix $\cS_k(\la)$  has nonconstant rank with respect to $\la$
 (see \cite[formula (2.14) and Theorem 2.7]{jE15}).
\begin{theorem}\label{globosc}
For problem \eqref{Sla},\eqref{E0} under assumptions \eqref{dnlaosc:E:Intro.4}
and \eqref{Smon} we have  formula \eqref{eig.in.int2}
presenting the number $\#\{\nu\in\sigma |a <\nu\le b\}$ of finite
eigenvalues in $(a,b].$

Moreover, under the assumption \eqref{finrank} the spectrum
$\sigma$ of problem \eqref{Sla},\eqref{E0} is bounded from below,
i.e., there exists $\la_{00}\in\bR$ such that
\begin{equation}\label{finspectr}
 \rank X_{N+1}^{[0]}(\lambda)=\rank X_{N+1}^{[0]}(\lambda^{-})\; \text{for all}\;\lambda<\lambda_{00},
\end{equation}
and
\begin{equation}\label{eig.glob}
    l_d(Y^{[0]}(b),0,N+1)-m+\sum\limits_{\nu\le b}\sum\limits_{k=0}^{N}\vartheta_{k}(\nu)=\#\{\nu\in\sigma |\nu\le b\},
  \end{equation}
where the sums $\sum\limits_{\nu\le b}\sum\limits_{k=0}^{N}\vartheta_{k}(\nu)$ and $\#\{\nu\in\sigma |\nu\le b\}:=\sum\limits_{\nu\le b}\theta(\nu)$ are finite and the constant $m$ is given by
\begin{equation}\label{m}
  m=l_d(Y^{[0]}(\la),0,N+1),\,\la<\min\{\la_0,\la_{00}\}.
\end{equation}
\end{theorem}
Theorem~\ref{globosc}  can be rewritten in terms of the principal
solution $Y_k^{[N+1]}(\la)$ at $k=N+1.$
\begin{theorem}\label{globosc*}
For problem \eqref{Sla},\eqref{E0} under assumptions \eqref{dnlaosc:E:Intro.4}
and \eqref{Smon} we have the following formula
  \begin{equation}\label{eig.in.int2*}
    l_d^{*}(Y^{[N+1]}(b),0,N+1)-l_d^{*}(Y^{[N+1]}(a),0,N+1)+\sum\limits_{a<\nu\le b}\sum\limits_{k=0}^{N}\vartheta_{k}(\nu)=\#\{\nu\in\sigma |a <\nu\le b\},
  \end{equation}
connecting the number $\#\{\nu\in\sigma |a <\nu\le b\}$ of finite
eigenvalues in $(a,b]$ with the number
$l_d^{*}(Y^{[N+1]}(\la),0,N)$ of backward focal points of the
principal solution  of \eqref{Sla} at $k=N+1$.

Moreover, under the assumption \eqref{finrank} we have
\begin{equation}\label{eig.glob*}
    l_d^{*}(Y^{[N+1]}(b),0,N+1)-m^{*}+\sum\limits_{\nu\le b}\sum\limits_{k=0}^{N}\vartheta_{k}(\nu)=\#\{\nu\in\sigma |\nu\le b\},
  \end{equation}
where the constant $m^{*}$ is given by
\begin{equation}\label{m*}
  m^{*}=l_d^{*}(Y^{[N+1]}(\la),0,N+1)=m=l_d(Y^{[0]}(\la),0,N+1),\,\la<\min\{\la_0,\la_{00}\}.
\end{equation}
\end{theorem}
\begin{proof}
Indeed, using relation \eqref{0-N+1} and replacing
$l(Y^{[0]}(\la),0,N+1)$ by $l^*(Y^{[N+1]}(\la),0,N+1)$ in
\eqref{eig.in.int2}, \eqref{eig.glob}, and \eqref{m} we derive
\eqref{eig.in.int2*}, \eqref{eig.glob*}, and \eqref{m*}.
\end{proof}

The main purpose of this paper is to present  renormalized versions
of Theorems~\ref{globosc},~\ref{globosc*}.
\section{Oscillation theory for differential Hamiltonian systems and renormalized oscillation theorems}\label{sec3}
Recall now some basic concepts of the oscillation theory of the
differential Hamiltonian systems \begin{equation}\label{Ham1}
     y'(t)=\cJ\cH(t)y(t),\,\cH(t)=\begin{pmatrix}
                           -C(t) & A^T(t) \\
                           A(t) & B(t) \\
                         \end{pmatrix},\,\cH(t)=\cH^T(t),
\end{equation}
where we assume that $\cH(t)$ is piecewise continuous with respect to $t\in[a,b]$ and the following \textit{Legendre} condition
\begin{equation}\label{B}
  B(t)\ge 0,\,t\in[a,b].
\end{equation}
Remark that the definition of a conjoined basis and the principal solution of \eqref{Ham1} can be introduced by analogy with the discrete case (see Section~\ref{sec2}).
\begin{definition}\label{multcon}
Assume \eqref{B}. We say (see \cite{Kratz-Analysis}) that
a point $t_0 \in (a, b]$ is a (left) proper focal
point of a conjoined basis $Y(t)=\bn{X(t)}{ U(t )}$ of \eqref{Ham1}, provided
\begin{equation}\label{propfocf}
    m_c(Y(t_0))=\deft X(t_0)-\deft X(t_0^-)=\rank X(t_0^-)-\rank X(t_0) \ge 1
\end{equation}
and $ m_c(Y(t_0))$ is its multiplicity.
\end{definition}
 Introduce the notation
\begin{equation}\label{notatfoc}
    l_c(Y,a,b)=\sum \limits_{\tau\in (a,b]} m_c(Y(\tau))
\end{equation}
for the total number of proper focal points (including the
multiplicities \eqref{propfocf}) in $(a,b].$ Recall that by the
Legendre condition $B(t)\ge 0$ all proper focal points are isolated
(see \cite[Theorem 3]{Kratz-Analysis}) then the sum
\eqref{notatfoc} is well-defined. Moreover, system \eqref{Ham1} is
called \textit{nonoscillatory} on $(-\infty,b]$  if for every
conjoined basis $Y(t)$ of \eqref{Ham1} the number of focal points
on $(-\infty,b]$ is finite (see \cite[Definition 2.3. and Theorem
2.2]{Hilsch4}, where we replace $\infty$ by $-\infty$).

Based on Definition~\ref{multcon} and
Remark~\ref{constInitConj}(ii) one can consider the numbers
\eqref{jump}
  as the multiplicities of
left proper focal points  of the conjoined basis $\cS_k(\la)\bnn=\binom{\cB_k(\la)}{\cD_k(\la)}$ of the Hamiltonian
differential system \eqref{PsiHam}, i.e., for any $a<b,\,a,b\in\bR$
\begin{equation} \label{conBproper}
    \vartheta_k(\la)=m_c(\cS_k(\la)\bnn),\quad \sum\limits_{a<\nu\le b}\vartheta_{k}(\nu)=l_c(\cS_k\bnn,a,b).
  \end{equation}
 Moreover, under
assumption \eqref{finrank} we have that system \eqref{PsiHam} is
nonoscillatory as $\la \rightarrow -\infty.$

By a similar way (see Remark~\ref{constInitConj}) one can connect
the notion of the algebraic multiplicity of finite eigenvalues of
\eqref{Sla},\eqref{E0} with the notion of the multiplicity of proper focal
points of the principal solution $Y_{N+1}^{[0]}(\la)$ of
\eqref{Sla}  considered as a function of $\la,$ i.e., we have
under \eqref{dnlaosc:E:Intro.4} and \eqref{Smon} that
  \begin{equation} \label{confineigproper}
    \theta(\la)=m_c(Y_{N+1}^{[0]}(\la)),\quad \#\{\nu\in\sigma |a <\nu\le b\}=l_c(Y_{N+1}^{[0]},a,b).
  \end{equation}

The main result in \cite[Theorem 2.2]{jE16}  connects the
multiplicities of proper focal points of conjoined bases of two
Hamiltonian systems under a majorant condition for their
Hamiltonians and under the Legendre condition assumed for one of
these systems. Here we reformulate this result in a new notation
convenient for the results of this paper.

\begin{theorem}[Comparison theorem for the continuous case]\label{maincompar}
Assume that
\begin{equation}\label{sympmatr}
    Y(\la):=Z(\la)\bnn,\,\hY(\la)=\hZ(\la)\bnn,
\end{equation}
where symplectic matrices
$Z(\la),\,\hZ(\la)\in C_p^1$ obey the conditions
\begin{equation}\label{majham}\begin{array}{c}
  \Psi(\hZ^{-1}(\la)Z(\la))=\hZ^T(\la) (\Psi(Z(\la)) -\Psi(\hZ(\la)))\hZ(\la) \ge 0,\\\phi(\hZ(\la)):=\bnn\Psi(\hZ(\la))\bnn \ge 0,\,\la\in[a,b].\end{array}
\end{equation}
Then for $\tY(\la):=\hZ^{-1}(\la)Y(\la)$ we have
\begin{equation}\label{elyscond51glob}
l_c(\hY,a,b)- l_c(Y,a,b)+l_c(\tY,a,b)=\mu(Y(a),\hat Y(a))-\mu(Y(b),\hat Y(b)).
\end{equation}
\end{theorem}
\begin{proof}
Here we rewrite the main result in \cite[Theorem 2.2]{jE16} replacing the variable $t$ by $\la$ and using
that $Y(\la)$ and $\hY(\la)$ are conjoined bases of  systems in form
\eqref{Ham1} with the Hamiltonians $\cH(\la)=\Psi(Z(\la)),\,\hH(\la)=\Psi(\hZ(\la)),$
respectively.
 In this notation
the proof follows from the proof of Theorem~2.2 in \cite{jE16}.
\end{proof}
As a corollary of Theorem~\ref{maincompar} one can derive the
following separation theorem (see \cite[Theorem 2.3]{jE16} and
\cite[Theorem~4.1]{pS.rSH17}) which we reformulate in the notation
introduced above.
\begin{theorem}[Separation theorem for the continuous case]\label{contsep}
Under the notation and the assumptions of Theorem~\ref{maincompar} suppose additionally that the matrix $\hZ^{-1}(\la)Y(\la)$ does not depend on $\la\in\bR,$ then
\begin{equation}\label{elyscond51sep}
l_c(\hY,a,b)- l_c(Y,a,b)=\mu(Y(\la),\hat Y(\la))|_b^a.
\end{equation}
\end{theorem}
In the subsequent results we will use the following corollary to Theorem~\ref{contsep}. Remark that this property is well-known, see, for example, \cite[Lemma 3.6]{BHK}.
\begin{corollary}\label{kratzcor}
Assume that the symmetric matrix $Q(\la)\in C_p^1$ is nondecreasing, i.e., $\dot{Q}(\la)\ge 0,\,\la\in
[a,b].$  Then $\rank Q(\la)$ is piecewise constant and
\begin{equation}\label{rankQind}
 l_c(({Q}\;{I})^T,a,b)= \sum\limits_{a<\la\le b}(\rank Q(\la^-)-\rank Q(\la))=\ind Q(a)-\ind Q(b).
\end{equation}
\end{corollary}
\begin{proof}
Applying Theorem~\ref{contsep} to the case $\hZ(\la):=\mmatrix{I & Q(\la) \\ 0 & I},\,Z(\la):=\hZ(\la)\cJ$ we see that
$\Psi(Z(\la))=\Psi(\hZ(\la))=\diag\{0,\dot{Q}(\la)\}\ge 0.$
 By \eqref{elyscond51sep} for $\hY(\la):=\hZ(\la)\bnn$ and $Y(\la):=Z(\la)\bnn$ we derive
 $ l_c(({Q}\;{I})^T,a,b)=\mu(({I}\;{0})^T,({Q(\la)}\;I)^T)|_b^a=\ind Q(\la)|_b^a$
or \eqref{rankQind}.
\end{proof}

Another important corollary to Theorem~\ref{maincompar} is the following theorem.
\begin{theorem}\label{forward}
 Consider the discrete symplectic system
\begin{equation}\label{gensys}
y_{k+1}(\la)=W_k(\la)y_k(\la),\,W_k^T(\la)\cJ W_k(\la)=\cJ,\,k\in[0,N]_{\sZ},
\end{equation}
such that the symplectic matrix $W_k(\la)\in C_p^1$ obeys the condition
\begin{equation}\label{PsiWk}
    \Psi(W_k(\la))\ge 0,\,k\in[0,N]_{\sZ},\,\la\in\bR,
\end{equation}
where the symmetric operator  $\Psi(\cdot)$ is defined by
\eqref{Smon}. Assume that $\cZ_{k}(\la)$ is a symplectic fundamental
matrix of \eqref{gensys} with $\cZ_0(\la)\in C_p^1$ and  condition \eqref{monZ} holds for $k=0$.
Then for a conjoined basis $\cY_{k}(\la)$  of \eqref{gensys} such
that $\cY_{0}(\la)=\cZ_{0}(\la)\bnn$ we have
\begin{align} \label{loccon}
 - \Delta l_c(\cY_k,a,b) +l_c(W_k\bnn,a,b)&=\mu(\cY_{k+1}(\la),W_k(\la)\bnn)|_b^a\\&=m_d(\cY_k(a))-m_d(\cY_k(b)).\notag
\end{align}
Moreover,
\begin{align}\label{locglob}
 l_c(\cY_{N+1},a,b)&-l_c(\cY_0,a,b) -\sum\limits_{k=0}^{N} l_c(W_k\bnn,a,b)\\&=l_d(\cY(b),0,N+1)-l_d(\cY(a),0,N+1).\notag
\end{align}
\end{theorem}
\begin{proof}
The condition  $\Psi(\cZ_0(\la))\ge 0,\,\la\in\bR$ and \eqref{PsiWk}
imply \eqref{monZ} for $k=1,\dots,N+1$ according
Proposition~\ref{psiZmod}(ii). Then, to prove \eqref{loccon} one
can apply Theorem~\ref{maincompar} for the case $\hat
Z(\la):=W_k(\la)$, $Z(\la):=\cZ_{k+1}(\la),$ and then $\tZ(\la):=\hat
Z^{-1}(\la)Z(\la)=\cZ_{k}(\la)$ for $k\in[0,N]_{\bZ}.$ It is clear that
under this settings assumptions \eqref{majham} of
Theorem~\ref{maincompar} are satisfied. Applying
\eqref{elyscond51glob} and using \eqref{forwind} for the
comparative indices in the right hand side we derive identity
\eqref{loccon}. Summing \eqref{loccon} from $k=0$ to $k=N$ we
derive \eqref{locglob}. The proof is completed.
\end{proof}
\begin{remark}\label{permutorder}
\par(i)For the special case $\cY_k(\la):=Y_k^{[0]}(\la)$  putting
$W_k(\la)\equiv\cS_k(\la),\,k\in[0,N]_{\sZ}$ we have in \eqref{locglob}
that $l_c(Y_0^{[0]},a,b)=0$ while according to
Definition~\ref{dnlaosc:D:finite.eigenvalue} the quantity
$l_c(Y_{N+1}^{[0]},a,b)$ is equal to the number $\#\{\nu\in\sigma
|a <\nu\le b\}:=\sum\limits_{a<\nu\le b}\theta(\nu)$ of finite
eigenvalues of \eqref{Sla},\eqref{E0} in $(a,b].$
\par(ii) Moreover, using the
finiteness of the sum $\sum\limits_{a<\nu\le b}\vartheta_{k}(\nu)$
we have
   $ \sum\limits_{a<\nu\le b}\sum\limits_{k=0}^{N}\vartheta_{k}(\nu)=\sum\limits_{k=0}^{N}\sum\limits_{a<\nu\le b}\vartheta_{k}(\nu)$
and under assumption \eqref{finrank} we have a similar property as
$a\rightarrow -\infty,$ i.e.,
   $ \sum\limits_{\nu\le b}\sum\limits_{k=0}^{N}\vartheta_{k}(\nu)=\sum\limits_{k=0}^{N}\sum\limits_{\nu\le b}\vartheta_{k}(\nu).$
So we see that Theorem~\ref{globosc} follows from
Theorem~\ref{forward}.
\end{remark}

Applying  Theorem~\ref{forward} we also show that for the
calculation of the number $\#\{\nu\in\sigma|\,a<\nu\le b\}$ of
finite eigenvalues of \eqref{Sla},\eqref{E0} it is possible to use the
transformation $\tY_{k}(\la)=R_{k}^{-1}Y_{k}^{[0]}(\la)$  of the
principal solution $Y_{k}^{[0]}(\la)$ of \eqref{Sla} with the
symplectic transformation matrix $R_{k}$ which does not depend on
$\la$ and obeys the condition
\begin{equation}\label{right}
    R_{N+1}=I.
\end{equation}
Then, by Remark~\ref{constInitConj}(iii) $\tY_{k}(\la)$ is a
conjoined of the transformed symplectic system \eqref{tSla} with
condition \eqref{tSmon}. Moreover, according to
Remark~\ref{constInitConj}(ii), for the block $\tcB_k(\la)$ of
$\tS_k(\la)$ in \eqref{tSla} there exists the finite limit
\begin{equation}\label{jumpt}
 \vartheta(\tS_k(\lambda_0))= \tilde\vartheta_k(\lambda_0) := \rank \tB_{k}(\lambda_0^{-})-\rank\tB_{k}(\lambda_0),\quad k\in[0,N]_\sZ,\,\la_{0}\in\bR.
\end{equation}
Introduce the notation
\begin{equation}\label{newnot1}
  n_R(\la_0):=\sum\limits_{k=0}^N\tilde\vartheta_k(\lambda_0),
\end{equation}
where  the index $R$ denotes the transformation matrix $R_k$ in the definition of $\tS_k(\la).$
We have the following result.

\begin{theorem}\label{newmay2019}
Assume \eqref{dnlaosc:E:Intro.4} and \eqref{Smon}, then for the
conjoined basis $\tY_{k}(\la)=R_{k}^{-1}Y_{k}^{[0]}(\la)$ of
\eqref{tSla}, where the symplectic matrix $R_{k}$ obeys condition
\eqref{right} we have the following formula
\begin{equation}\label{renormeigmay2019}\begin{aligned}
    l_d(R^{-1}Y^{[0]}(b),0,N+1)&-l_d(R^{-1}Y^{[0]}(a),0,N+1)\\&+\sum\limits_{a<\nu\le b}n_R(\nu)=\#\{\nu\in\sigma|\,a<\nu\le b\}
 \end{aligned} \end{equation}
connecting the number of finite eigenvalues of \eqref{Sla},\eqref{E0} in
$(a,b]$ with the number of forward focal points of the conjoined
basis $\tY_{k}(\la)=R_{k}^{-1}Y_{k}^{[0]}(\la)$ and the numbers given by \eqref{jumpt} and \eqref{newnot1}.

Moreover, under assumption \eqref{finrank} the sum
$\sum\limits_{a<\nu\le b}n_R(\nu)$ is finite as $a\rightarrow -\infty,$ then  for any $a<\la_{1}=\min \sigma$ the left-hand side of \eqref{renormeigmay2019} presents the number $\#\{\nu\in\sigma|\,\nu\le b\}$ of finite eigenvalues of \eqref{Sla},\eqref{E0} less then or equal to $\la:=b.$
\end{theorem}

\begin{proof}
Applying \eqref{locglob} to the case $\cY_k(\la):=\tilde
Y_k(\la)=R_{k}^{-1}\,Y_k^{[0]}(\la)$ and $W_k(\la):=\tS_k(\la)$ we
have
\begin{equation}\label{ad4}
\begin{aligned}
 l_c(\tY_{N+1},a,b)&-l_c(\tY_{0},a,b) -\sum\limits_{k=0}^{N} l_c(\tS_k\bnn,a,b)
 \\&=l_d(\tY(b),0,N+1)-l_d(\tY(a),0,N+1).
\end{aligned}
\end{equation}
We see that $\tY_{0}=R_{0}^{-1}Y_{0}^{[0]}(\la)=R_{0}^{-1}\bnn$
does not depend on $\la,$ then $l_c(\tY_{0},a,b)=0,$ while
$\tY_{N+1}(\la)=R_{N+1}^{-1}Y_{N+1}^{[0]}(\la)=Y_{N+1}^{[0]}(\la)$
due to condition \eqref{right}. Then, according to \eqref{conBproper} (see also Remark~\ref{permutorder}(ii)) we have
 $\sum\limits_{k=0}^{N} l_c(\tS_k\bnn,a,b)=\sum\limits_{a<\nu\le b}n_R(\nu).$ Similarly,
$l_c(Y_{N+1}^{[0]},a,b)=\#\{\nu\in\sigma|\,a<\nu\le b\}$ by \eqref{confineigproper}.
Substituting these equalities into \eqref{ad4} we complete the
proof of \eqref{renormeigmay2019}.

As in the proof of \cite[Corollary 2.5]{jE15} we have from
\eqref{renormeigmay2019}
\begin{equation}\label{ad3}
\begin{aligned}
&\left|\sum\limits_{a<\nu\le b}n_R(\nu) -\#\{\nu\in\sigma|\,a<\nu\le b\}\right|\\&=\left|l_d(R^{-1}Y^{[0]}(a),0,N+1)-l_d(R^{-1}Y^{[0]}(b),0,N+1)\right|\le n(N+1),
\end{aligned}
\end{equation}
where we have used estimate \eqref{eatnum}. So we see from
\eqref{ad3} that the sum $\sum\limits_{a<\nu\le b}n_R(\nu)$ is
finite as $a\rightarrow -\infty$ if and only if the spectrum
$\sigma$ of \eqref{Sla},\eqref{E0} is bounded from below. But
condition \eqref{finrank} is equivalent to \eqref{finspectr}, see
\cite[Corollary 2.5]{jE15}, then $\#\{\nu\in\sigma|\,a<\nu\le b\}$
is bounded for any fixed $b\in\bR$ as $a\rightarrow -\infty.$
Finally \eqref{renormeigmay2019} for $a<\la_{1}=\min \sigma$
 calculates the number of finite eigenvalues
 of \eqref{Sla},\eqref{E0} less then or equal to $b$.
\end{proof}
For the special choice of the transformation matrix $R_{k}$ in form
    $R_{k}:=Z_{k}^{[N+1]}(a)$
or
   $ R_{k}:=Z_{k}^{[N+1]}(b)$
in the formulation of Theorem~\ref{newmay2019} we derive the
renormalized oscillation theorem for  forward focal points.
\begin{theorem}\label{globoscrel1}
Let $Z_{k}^{[N+1]}(\la)$ be symplectic fundamental matrix of
\eqref{Sla} such that $Z_{N+1}^{[N+1]}(\la)=I,\,\la\in\bR.$ Then,
under the assumptions of Theorem~\ref{newmay2019} for $R_{k}:=Z_{k}^{[N+1]}(a)$ formula \eqref{renormeigmay2019}  takes the form
\begin{equation}\label{renormeiga}\begin{aligned}
    l_d((Z^{[N+1]}(a))^{-1}Y^{[0]}(b),0,N+1)+\sum\limits_{a<\nu\le b}n_{Z^{[N+1]}(a)}(\nu)=\#\{\nu\in\sigma|\,a<\nu\le b\}.
 \end{aligned} \end{equation}

Similarly, if  $R_{k}:=Z_{k}^{[N+1]}(b)$, then instead of \eqref{renormeigmay2019}  we have
\begin{equation}\label{renormeigb}\begin{aligned}
    -l_d((Z^{[N+1]}(b))^{-1}Y^{[0]}(a),0,N+1)+\sum\limits_{a<\nu\le b}n_{Z^{[N+1]}(b)}(\nu)=\#\{\nu\in\sigma|\,a<\nu\le b\}.
\end{aligned}  \end{equation}
If, additionally, \eqref{finrank} holds, then for  any $a<\la_{1}=\min \sigma$ formulae \eqref{renormeiga} and \eqref{renormeigb} present the number $\#\{\nu\in\sigma|\,\nu\le b\}$ of finite eigenvalues of \eqref{Sla},\eqref{E0} less then or equal to $\la:=b.$
\end{theorem}
\begin{proof}
For the special choice $R_k:=Z_k^{[N+1]}(\beta),\,\beta\in\{a,b\}$  in \eqref{renormeigmay2019} we have
$l_d((Z^{[N+1]}(\beta))^{-1}Y^{[0]}(\beta),0,N+1)=0$ because $Z_k^{[N+1]}(\beta)$ and $Z_k^{[0]}(\beta)$ solve the same symplectic system \eqref{Sla} for $\la:=\beta,\,\beta\in\{a,b\}.$  Then, all assertions of Theorem~\ref{globoscrel1} follow from Theorem~\ref{newmay2019}.
\end{proof}

\begin{example}\label{exampl1}
Consider an  example given in \cite[Example 2.9]{jE15} which illustrated
 applications of formula \eqref{eig.in.int2} in Theorem~\ref{globosc}. Here we use this example to illustrate applications of \eqref{renormeiga} in Theorem~\ref{globoscrel1}. Consider problem \eqref{Sla},\eqref{E0} for the trigonometric
difference system  with
\begin{equation}\label{matrtrig}
    \cS_{k}(\lambda)=\begin{pmatrix}
                       \cos(\lambda) & \sin(\lambda) \\
                       -\sin(\lambda) & \cos(\lambda) \\
                     \end{pmatrix},\,k\in[0,N]_{\bZ}.
\end{equation}
We have $\Psi(\cS_{k}(\lambda))=I\ge 0,$ then the monotonicity
condition \eqref{Smon} holds for $\lambda\in \mathbb{R}.$ The
principal solution of \eqref{Sla},\eqref{E0} with \eqref{matrtrig}
has the form
$Y_k^{[0]}(\lambda)=[\sin(k\lambda)\,\cos(k\lambda)]^{T},$ then the
finite eigenvalues of this problem $\lambda_{p}=\pi p/(N+1),\,p\in
\mathbb{Z}.$ Point out that the spectrum $\sigma$ of this problem
is unbounded. However one can use
Theorems~\ref{globosc},~\ref{globoscrel1} to calculate the number
of finite eigenvalues in any half-open interval $(a,b].$ The
multiplicities of focal points of $Y^{[0]}_{k}(\la)$ in $(k,k+1]$
are given by
\begin{equation}\label{ad2}
    m_{d}(Y_k^{[0]}(\lambda))=\left\{
           \begin{array}{ll}
                          1, & \hbox{$\lambda=\pi p/(k+1),\,\lambda\neq\pi l,\,p,l\in\mathbb{Z}$;} \\
             1, & \hbox{$\sin(\lambda)\sin(k\lambda)\sin((k+1)\lambda)<0$;} \\
             0, & \hbox{$\text{otherwise}$.}
           \end{array}
         \right.
\end{equation}
So we see that
$m_{d}(Y_k^{[0]}(\lambda))=m_{d}(Y_k^{[0]}(\lambda+\pi l)),\,l\in
\mathbb{Z}$ and then $l_{d}(Y^{[0]}(\lambda),0,N+1)$ is periodic
with the minimal period $T=\pi$ and nondecreasing in any interval
$[a,b]\subseteq [\pi l,\pi(l+1)),\,l\in \mathbb{Z}.$ We have
\begin{equation*}\begin{array}{l}
l_{d}(Y^{[0]}(\lambda),0,N+1)=\left\{
                 \begin{array}{ll}
                   0, & \hbox{$\lambda\in[0,\pi/(N+1)$;} \\
                   1, & \hbox{$\lambda\in[\pi/(N+1),2\pi/(N+1)$;} \\
                   \dots & \hbox{$\dots$} \\
                   N, & \hbox{$\lambda\in [N\pi/(N+1),\pi)$;}
                 \end{array}
               \right.\\[2mm] l_{d}(Y^{[0]}(\lambda),0,N+1)=l_{d}(Y^{[0]}(\lambda+\pi l),0,N+1),\quad l\in
\mathbb{Z}.\end{array}\end{equation*}
Put
\begin{equation}\label{ad8}
  a=q\pi/(N+1),\,b=r\pi/(N+1),\quad r>q,\,r,q\in\mathbb{Z}.
\end{equation}
Then there are $r-q$ finite eigenvalues of \eqref{Sla},\eqref{E0} located in $(a,b].$
For example, if $r=N,\,q=N-1,\,N>0$ then, in \eqref{eig.in.int2}
$$\sum\limits_{a<\nu\le b}\sum\limits_{k=0}^{N}\vartheta_{k}(\nu)=0,
\;l_d(Y^{[0]}(b),0,N+1)=N, \; l_d(Y^{[0]}(a),0,N+1)=N-1,$$ and the
number $\#\{\nu\in\sigma |a <\nu\le b\}$ of finite eigenvalues in
$(a,b]$ is equal to~$1.$  Note that to apply \eqref{eig.in.int2} we
have to calculate $\#\{\nu\in\sigma |a <\nu\le b\}$ using the pair
$Y^{[0]}(a),\,Y^{[0]}(b)$ of highly oscillatory (for $N\gg 1$)
conjoined bases of \eqref{Sla}. In contrast, applying
\eqref{renormeiga} in Theorem~\ref{globoscrel1} we deal with only
one slowly oscillatory conjoined basis $\tY_k(b)$ of \eqref{tSla}
with $R_k:=Z_k^{[N+1]}(a).$ Indeed,
 the matrix $ Z_k^{[N+1]}(\la),\,k=0,\dots,N+1$ has the form
\begin{equation*}
  Z_k^{[N+1]}(\la)=\begin{pmatrix}
                       \cos((N-k+1)\lambda) & \;-\sin((N-k+1)\lambda) \\
                       \sin((N-k+1)\lambda) &\; \cos((N-k+1)\lambda) \\
                     \end{pmatrix},
\end{equation*}
then
\begin{equation}\label{ad1}
   \tY_k(b)=(Z_k^{[N+1]}(a))^{-1}Y^{[0]}(b)=\binom{\sin(\beta_k)}{\cos(\beta_k)},\,\beta_k=(b-a)k+(N+1)a
\end{equation}
and by analogy with \eqref{ad2} we have
 \begin{equation}\label{ad7}
    m_{d}(\tY_k(b))=\left\{
           \begin{array}{ll}
                          1, & \hbox{$\beta_{k+1}=\pi p,\,b-a\neq\pi l,\,p,l\in\mathbb{Z}$;} \\
             1, & \hbox{$\sin(b-a)\sin(\beta_k)\sin(\beta_{k+1})<0$;} \\
             0, & \hbox{$\text{otherwise}$,}
           \end{array}
         \right.\quad k\in [0,N]_{\sZ},
\end{equation}
where $\beta_k$ is given by \eqref{ad1}.
The matrix $\tS_k(\la)$ associated with $R_k:=Z_k^{[N+1]}(a)$ has the form
 $ \tS_k(\la)=\cS_k^{-1}(a)\cS_k(\la)$, then $\tcB_k(\la)=\sin(\la-a).$ We see that for the case
 $a=\pi(N-1)/(N+1),\,b=\pi N/(N+1)$ the matrix
 $\tB_k(\la)$ has constant rank for $\la\in(a,b]$, i.e.,
 $\sum\limits_{a<\nu\le b}n_{Z^{[N+1]}(a)}(\nu)=0,$ while according to \eqref{ad7} $m_d(\tY_N(b))=~1,\,m_d(\tY_k(b))=~0,\,k\ne N.$ Finally, in this case $l_d(\tY(b),0,N+1)=1$ calculates the number of finite eigenvalues in $(a,b].$

Consider another situation when the block of $\cB_k(\la)$ of the matrix $\cS_k(\la)$ has nonconstant rank in $(a,b].$ Assume that $a,\,b$ are given by \eqref{ad8}, where $q=N$ and $r=N+2,$ then $\rank\cB_k(\la)=0$ for $\la=\pi\in(a,b]$ and for the given case $\sum\limits_{a<\nu\le b}\sum\limits_{k=0}^{N}\vartheta_{k}(\nu)=N+1,$ $l_d(Y^{[0]}(b),0,N+1)=1$, $l_d(Y^{[0]}(a),0,N+1)=N,$ then, according to \eqref{eig.in.int2} we  have the correct result $1-N+N+1=2$ for the number of finite eigenvalues of \eqref{Sla},\eqref{E0}. Applying Theorem~\ref{globoscrel1} instead of Theorem~\ref{globosc} we again have that $\tB_k(\la)=\sin(\la-a)$ has constant rank in $(a,b],$ i.e., $\sum\limits_{a<\nu\le b}n_{Z^{[N+1]}(a) }(\nu)=0,$ while according to \eqref{ad7} $m_d(\tY_N(b))=1,$ and $m_d(\tY_l(b))=1$ for $l=\lfloor N/2 \rfloor.$
Finally, as before $l_d(\tY(b),0,N+1)=2$ calculates the number of finite eigenvalues in $(a,b].$
So we see that in this situation we deal with the slowly oscillatory conjoined basis $\tY_k(b)$ and the nonoscillatory coefficient matrix $\tS_k(\la).$
\end{example}

By a similar way one can prove an
analog of Theorem~\ref{forward} for the time-reversed system
\eqref{gensys}.
\begin{theorem}\label{backward}
 Consider  symplectic system \eqref{gensys} under assumption \eqref{PsiWk}.
Suppose that for $k=N+1$ a symplectic fundamental matrix
$\cZ_{k}(\la)$ of \eqref{gensys} is piecewise continuously
differentiable and obeys condition \eqref{monZ*}.
Then  for a conjoined basis of \eqref{gensys}, such that
$\cY_{k}(\la):=\cZ_{k}(\la)\bnn$ we have
\begin{align} \label{loccon*}
  \Delta l_c(\cY_k,a,b) +l_c(W_k\bnn,a,b)&=\mu^*(\cY_{k}(\la),W_k^{-1}(\la)\bnn)|_b^a\\&=m_d^*(\cY_k(a))-m_d^*(\cY_k(b)).\notag
\end{align}
Moreover,
\begin{align}\label{locglob*}
 l_c(\cY_{N+1},a,b)&-l_c(\cY_0,a,b) +\sum\limits_{k=0}^{N} l_c(W_k\bnn,a,b)\\&=l_d^*(\cY(a),0,N+1)-l_d^*(\cY(b),0,N+1).\notag
\end{align}
\end{theorem}
\begin{proof}
By Proposition~\ref{psiZmod}(iii) we have  that conditions
\eqref{monZ*}, \eqref{monZP*} hold for $\cZ_{k}(\la)$ with $k\in
[0,N+1]_{\sZ}.$ Then, we prove \eqref{loccon*} applying
Theorem~\ref{maincompar} to the case $\hat
Z(\la):=P_3W_k^{-1}(\la)P_3$, $Z(\la):=P_3\cZ_{k}(\la)P_3,$ and
then $\tZ(\la):=\hat Z^{-1}(\la)Z(\la)=P_3\cZ_{k+1}(\la)P_3$ for
$k=0,\dots,N,$ where the matrix $P_3$ is given by \eqref{dualrep}.
It is clear that under this settings assumptions \eqref{majham} of
Theorem~\ref{maincompar} are satisfied. Note also that by
Definition~\ref{multcon} and the relation $W_k^{-1}(\la)=-\cJ
W_k^{T}(\la)\cJ$ we have
$l_c(P_3W^{-1}_kP_3\bnn,a,b)=l_c(W_k\bnn,a,b).$ Applying
\eqref{elyscond51glob} and using \eqref{dualrep} and
\eqref{backwind} for the comparative indices
$\mu(P_{3}\cY_{k}(\la),P_{3}W_{k}^{-1}(\la)P_3\bnn)=\mu^{*}(\cY_{k}(\la),W_{k}^{-1}(\la)\bnn)=m_{d}^{*}(\cY_{k}(\la))$
for $\la=a$ and $\la=b$ in the right hand side of
\eqref{elyscond51glob} we derive identity \eqref{loccon*}.

Summing \eqref{loccon*} from $k=0$ to $k=N$ we derive
\eqref{locglob*}. The proof is completed.
\end{proof}

Consider  the
transformation $\tY_{k}(\la)=R_{k}^{-1}Y_{k}^{[N+1]}(\la)$  of the
principal solution $Y_{k}^{[N+1]}(\la)$ of \eqref{Sla} at $k=N+1$ with the
symplectic transformation matrix $R_{k}$ which does not depend on
$\la$ and obeys the condition
\begin{equation}\label{left}
    R_{0}=I.
\end{equation}
Using the notation \eqref{jumpt} and \eqref{newnot1} associated with the transformation matrix $R_k$ with property \eqref{left} one can formulate the following analog of Theorem~\ref{newmay2019} for backward focal points.

\begin{theorem}\label{newmay2019*}
Assume \eqref{dnlaosc:E:Intro.4} and \eqref{Smon}, then for the
conjoined basis $\tY_{k}(\la)=R_{k}^{-1}Y_{k}^{[N+1]}(\la)$ of
\eqref{tSla}, where the symplectic matrix $R_{k}$ obeys condition
\eqref{left} we have the following formula
\begin{equation}\label{renormeigmay2019*}\begin{aligned}
    l_d^*(R^{-1}Y^{[N+1]}(b),0,N+1)&-l_d^*(R^{-1}Y^{[N+1]}(a),0,N+1)\\&+\sum\limits_{a<\nu\le b}n_R(\nu)=\#\{\nu\in\sigma|\,a<\nu\le b\}.
 \end{aligned} \end{equation}
connecting the number of finite eigenvalues of \eqref{Sla},\eqref{E0} in
$(a,b]$ with the number of backward focal points of the conjoined
basis $\tY_k(\la)$ and the numbers given by \eqref{jumpt} and \eqref{newnot1}.

Moreover, under assumption \eqref{finrank} the sum
$\sum\limits_{a<\nu\le b}n_R(\nu)$ is finite as $a\rightarrow -\infty,$ then  for any $a<\la_{1}=\min \sigma$ the left-hand side of \eqref{renormeigmay2019*} presents the number $\#\{\nu\in\sigma|\,\nu\le b\}$ of finite eigenvalues of \eqref{Sla},\eqref{E0} less then or equal to $\la:=b.$
\end{theorem}

\begin{proof}
Applying \eqref{locglob*} to the case $\cY_k(\la):=\tilde
Y_k(\la)=R_{k}^{-1}\,Y_k^{[N+1]}(\la)$ and $W_k(\la):=\tS_k(\la)$ we
have
\begin{equation}\label{ad4*}
\begin{aligned}
 l_c(\tY_{N+1},a,b)&-l_c(\tY_{0},a,b) +\sum\limits_{k=0}^{N} l_c(\tS_k\bnn,a,b)
 \\&=l_d^*(\tY(a),0,N+1)-l_d^*(\tY(b),0,N+1).
\end{aligned}
\end{equation}
We see that $\tY_{N+1}=R_{N+1}^{-1}Y_{N+1}^{[N+1]}(\la)=R_{N+1}^{-1}\bnn$
does not depend on $\la,$ then $l_c(\tY_{N+1},a,b)=0,$ while
$\tY_{0}(\la)=R_{0}^{-1}Y_{0}^{[N+1]}(\la)=Y_{0}^{[N+1]}(\la)$
due to condition \eqref{left}. According to \eqref{conBproper} we have
 $\sum\limits_{k=0}^{N} l_c(\tS_k\bnn,a,b)=\sum\limits_{a<\nu\le b}n_R(\nu).$ Similarly,
$l_c(Y_{0}^{[N+1]},a,b)=l_c(Y_{N+1}^{[0]},a,b)=\#\{\nu\in\sigma|\,a<\nu\le
b\}$ because of \eqref{confineigproper} and the property
$ - X_0^{[N+1]}(\la)=X_{N+1}^{[0]\;T}(\la),\,\la\in\bR$
which follows from the Wronskian identity $$w(Y_k^{[0]}(\la),Y_k^{[N+1]}(\la))=Y_k^{[0]\;T}(\la)\cJ Y_k^{[N+1]}(\la)=\const$$ for the principal solutions $Y_k^{[0]}(\la),\,Y_k^{[N+1]}(\la)$ of \eqref{Sla}.

Substituting these equalities into \eqref{ad4*} we complete the
proof of \eqref{renormeigmay2019*}.

As in the proof of Theorem~\ref{newmay2019} we have from
\eqref{renormeigmay2019*}
\begin{equation*}
\begin{aligned}
&\left |\sum\limits_{a<\nu\le b}n_R(\nu) -\#\{\nu\in\sigma|\,a<\nu\le b\}\right|\\&=\left|l_d^*(R^{-1}Y^{[0]}(a),0,N+1)-l_d^*(R^{-1}Y^{[0]}(b),0,N+1)\right|\le n(N+1),
\end{aligned}
\end{equation*}
where we have used estimate \eqref{eatnum}. Repeating the same
arguments as in the proof of Theorem~\ref{newmay2019} we complete
the proof of Theorem~\ref{newmay2019*}.
\end{proof}
For the special choice of the transformation matrix
    $R_{k}:=Z_{k}^{[0]}(a)$
or
   $ R_{k}:=Z_{k}^{[0]}(b)$
in the formulation of Theorem~\ref{newmay2019*} we derive the
renormalized oscillation theorem for  backward focal points.

\begin{theorem}\label{globoscrel1*}
Under the assumptions of Theorem~\ref{newmay2019*} consider the transformation matrix $R_{k}:=Z_{k}^{[0]}(a)$. Then formula \eqref{renormeigmay2019*} takes the form
\begin{equation}\label{renormeiga*}
    l^*_d((Z^{[0]}(a))^{-1}Y^{[N+1]}(b),0,N+1)+ \sum\limits_{a<\nu\le b}n_{Z^{[0]}(a)}(\nu)=\#\{\nu \in\sigma|\,a<\nu\le b\},
  \end{equation}
while for the choice of $R_k$ in form $R_{k}:=Z_{k}^{[0]}(b)$ we have instead of \eqref{renormeigmay2019*}
\begin{equation}\label{renormeigb*}
   - l^*_d((Z^{[0]}(b))^{-1}Y^{[N+1]}(a),0,N+1)+ \sum\limits_{a<\nu\le b}n_{Z^{[0]}(b)}(\nu)=\#\{\nu \in\sigma|\,a<\nu\le b\}.
  \end{equation}

Moreover,  under assumption \eqref{finrank}  formulae \eqref{renormeiga*} and \eqref{renormeigb*} with $a<\la_1=\min \sigma$ present the number $\#\{\nu\in\sigma|\,\nu\le b\}$ of finite eigenvalues of \eqref{Sla},\eqref{E0} less then or equal to $\la:=b.$
\end{theorem}
\begin{proof} The proof follows from Theorem~\ref{newmay2019*}
using that
$l^*_d(Z^{[0]\,-1}(\beta)Y^{[N+1]}(\beta),0,N+1)=0,\,\beta\in\{a,b\}$
(see the proof of Theorem~\ref{globoscrel1}).
\end{proof}
\section{Relative oscillation numbers and renormalized oscillation theorems}\label{sec4}
In this section we present another approach to the results of the renormalized and relative oscillation theory based on discrete comparison theorems. This approach is presented in  \cite[Section 6.1]{bookDEH} under  restriction \eqref{rBconst}. Here we generalize the results from \cite[Section 6.1]{bookDEH} referred to the renormalized theory deleting condition \eqref{rBconst}.

For the first step in this direction we recall
  the discrete comparison theorem, see \cite[Theorem 2.1]{E-DE-2010}, \cite[Theorem 3.3]{D-E-JDEA} and a notion of the \textit{relative oscillation numbers} (see  \cite{E-ADV}, \cite[Definition 3.2]{D-E-JDEA}, \cite[Section~4.2]{bookDEH}).
Introduce the notation
\begin{gather}\label{langl}
    \langle  \cS\rangle=\begin{pmatrix}
                          \mathcal X \\
                          \mathcal U \\
                        \end{pmatrix},\,\mathcal X=\begin{pmatrix}
                                                     I & 0 \\
                                                     \cA & \cB \\
                                                   \end{pmatrix},\,\mathcal U=\begin{pmatrix}
                                                     0 & -I \\
                                                     \cC & \cD \\
                                                   \end{pmatrix},
    \quad  \cS=\begin{pmatrix}
                                 \cA & \cB \\
                                 \cC & \cD \\
                               \end{pmatrix}
\end{gather}
for  $ \cS\in \Sp$ separated into $n\times n$
blocks $\cA,\,\cB,\, \cC,\,\cD.$ In \cite[Lemma~2.3]{E-DE-2010} we
proved that $4n\times 2n$ matrices  $\langle \cS \rangle,\,\langle
\hS \rangle$ associated with  $\cS,\,\hS\in \Sp$ obey
\eqref{conjoined} (with $n$ replaced by $2n$) and then the
comparative index for the pair $\langle \cS \rangle,\,\langle \hS
\rangle$ is well defined. The main properties of the comparative index $\mu(\langle \cS\rangle,\langle \hS\rangle)$ are proved in \cite[Lemma 2.3]{E-DE-2010} (see also \cite[Section 3.3]{bookDEH}).

Recall  the notion of the {relative
oscillation numbers} for two symplectic difference systems.

\begin{definition}                              \label{relnum}
Let $Y_k,\,\hY_k$ be conjoined bases of the symplectic systems
\begin{equation}\label{twosys}
  Y_{k+1}=\cS_k Y_k,\quad \hY_{k+1}=\hS_k \hY_k
\end{equation}
 associated with symplectic fundamental matrices
$Z_k,\,\hZ_k$ such that \eqref{sympmatr} hold.

Then the relative oscillation number is defined as
\begin{equation}                                  \label{relnumb}
\begin{aligned}
    \#_k(\hZ, Z)&=\mu(\langle  \hZ_{k}^{-1}Z_{k}
\rangle,\langle  \hZ_{k+1}^{-1} Z_{k+1}\rangle)-\mu(\langle
\hS_{k} \rangle,\langle   \cS_{k} \rangle)
\\[2mm]
&=\mu(\langle \cS_{k} \rangle,\langle  \hS_{k} \rangle)-
\mu(\langle  \hZ_{k+1}^{-1}Z_{k+1}
\rangle,\langle  \hZ_{k}^{-1} Z_{k}\rangle) .
\end{aligned}
\end{equation}
\end{definition}

The following comparison result was proved in  \cite[Theorem
2.1]{E-DE-2010}.
\begin{theorem}                                 \label{compar}
Let $Y_k,\,\hY_k$ be conjoined bases of \eqref{twosys} associated
with symplectic fundamental matrices $Z_k,\,\hZ_k$ such that
\eqref{sympmatr} holds, then
\begin{equation}\label{var41}
l_d(\hY,0,N+1)-l_d(Y,0,N+1)+\mu(Y_{k},\hY_{k})|_0^{N+1}=\#(\hZ,Z,0,N),\end{equation}
where
\begin{equation}\label{globrelnu}
  \#(\hZ,Z,0,N)=\sum\limits_{k=0}^{N}\#_k(\hZ, Z)
\end{equation}
and $l_d(\hY, 0,N+1),\,l_d(Y,0,N+1)$ are the numbers of forward
focal points in $(0,N+1]$ defined by \eqref{E:n1.def}.
\end{theorem}

Applying Theorem~\ref{compar} to \eqref{Sla} for the  case $\hS_k:=\cS_k(b),\,\cS_k:=\cS_k(a)$ and
$\hZ_k:=Z_k^{[N+1]}(b),\,Z_k:=Z_k^{[0]}(a)$ we see that formula \eqref{var41}
can be rewritten in the form (see \cite[Corollary 2.4]{E-DE-2010})
\begin{equation}\label{princ}
  l_d(Y^{[0]}(b),0,N+1)-l_d(Y^{[0]}(a),0,N+1)=\#(Z^{[N+1]}(b),Z^{[0]}(a),0,N).
\end{equation}
Substituting \eqref{princ} into the left-hand side of
\eqref{eig.in.int2} we derive
 \begin{equation}\label{rel1}
    \#(Z^{[N+1]}(b),Z^{[0]}(a),0,N)+\sum\limits_{a<\mu\le b}\sum\limits_{k=0}^{N}\vartheta_{k}(\mu)=\#\{\nu\in\sigma |a <\nu\le b\},
  \end{equation}
  where the relative oscillation numbers $\#(Z^{[N+1]}(b),Z^{[0]}(a),0,N)$ are given by \eqref{relnumb} and \eqref{globrelnu}.

Main questions which we answer in  this section are connected with simplifications of the sum in the left-hand side of \eqref{rel1} which leads to the proof of formula \eqref{renormeiga2} and the central result of this section, Theorem~\ref{globoscbig}.
In particular, we investigate sufficient conditions  for the \textit{majorant} condition (see \cite[Section 2]{D-E-JDEA})
 \begin{equation}\label{maj}
   \mu(\langle
\cS_{k}(b) \rangle,\langle   \cS_{k}(a) \rangle)=0
 \end{equation}
 which makes the relative oscillation numbers nonnegative
\begin{equation}\label{relnon}
\#(Z^{[N+1]}(b),Z^{[0]}(a),0,N)\ge 0.
\end{equation}

We begin with the following interesting
oscillation result which is  derived as a corollary to
Theorem~\ref{contsep}.
\begin{lemma}\label{newosc1}
For a symplectic matrix
$W(\la)\in C_p^1$ assume that
\begin{equation}\label{psimon}
  \Psi(W(\la))\ge 0,\,\la\in\bR.
\end{equation}
Then for any $\la\in\bR$ and for  arbitrary   constant
matrix $\hat W\in \Sp$
 there exists the limit $\rank(\hat W-W(\la^-)).$ For the nonnegative numbers
\begin{equation}\label{completrank}
  \rho_{\hat W}(W(\la)):=\rank(\hat W-W(\la^-))-\rank(\hat W-W(\la))
\end{equation}
and
\begin{equation}\label{rB}
  \vartheta(W(\la)):=\rank B(\la^-)-\rank B(\la),\quad W(\la)=\smat{A(\la)}{B(\la)}{C(\la)}{D(\la)}
\end{equation}
we have the connection
\begin{equation}\label{conbig}
  \sum\limits_{a<\nu\le b}\rho_{\hat W}(W(\nu))-\sum\limits_{a<\nu\le b}\vartheta(W(\nu))=\mu(\langle W(\la)\rangle,\langle \hat W\rangle)|_b^a,
\end{equation}
where the matrix $\langle W(\la)\rangle$ is defined by
\eqref{langl}.
\end{lemma}
The proof of Lemma~\ref{newosc1} is postponed to Appendix~\ref{sec5}.
For the subsequent proofs we need the following properties of the numbers $\rho_{\hat W}(W(\la))$ given by \eqref{completrank}.
\begin{proposition}\label{newosc}
Under the assumptions and the notation of Lemma~\ref{newosc1} the following properties of numbers \eqref{completrank} hold.
\par(i) Let $R,\,P$ be arbitrary symplectic matrices which do not depend on $\la,$ then
\begin{equation}\label{prop1}
  \rho_{R^{-1}\hat W P}(R^{-1}W(\la)P)=\rho_{\hat W}(W(\la)),\quad \la\in\bR.
\end{equation}
In particular, for  the coefficient matrices $\cS_k(\la),\,\tS_k(\la)$   of systems \eqref{Sla} and \eqref{tSla} we have for any fixed $\beta\in\bR$
\begin{equation}\label{prop2}
  \rho_{\tS_k(\beta)}(\tS_k(\la))=\rho_{\cS_k(\beta)}(\cS_k(\la)),\quad \la\in\bR.
\end{equation}
\par (ii) For any $a<b,\,a,b\in\bR$ there exist the limits
$\rank(W(a)-W(\la^-)),$ $\rank(W(b)-W(\la^-))$ and for the numbers
\begin{equation}\label{completrankWa}
  \rho_{W(a)}(W(\la)):=\rank(W(a)-W(\nu))|_{\la}^{\la^{-}},\;\;\rho_{W(b)}(W(\la)):=\rank(W(b)-W(\nu))|_{\la}^{\la^{-}}
\end{equation}
and $\vartheta(W(\la))$ defined by \eqref{rB} we have
\begin{equation}\label{conbigWa}
  \sum\limits_{a<\nu\le b}\rho_{W(a)}(W(\nu))-\sum\limits_{a<\nu\le b}\vartheta(W(\nu))=-\mu(\langle W(b)\rangle,\langle W(a)\rangle),
\end{equation}

\begin{equation}\label{conbigWb}
  \sum\limits_{a<\nu\le b}\rho_{W(b)}(W(\nu))-\sum\limits_{a<\nu\le b}\vartheta(W(\nu))=\mu(\langle W(a)\rangle,\langle W(b)\rangle).
\end{equation}
\par(iii) For the numbers \eqref{completrankWa} we have the connection
\begin{equation*}\label{conWbWa}
    \sum\limits_{a<\nu\le b}\rho_{W(b)}(W(\nu))-\sum\limits_{a<\nu\le b}\rho_{W(a)}(W(\nu))=\rank(W(b)-W(a)).
\end{equation*}
\end{proposition}
\begin{proof}
The proof of (i) follows from the definition of numbers \eqref{completrank}. Indeed, using the nonsingularity of the symplectic matrices $R,\,P$ we have
\begin{align*}
  \rho_{R^{-1}\hat WP}(R^{-1}W(\la)P)&:=\rank(R^{-1}(\hat W-W(\la^-))P)-\rank(R^{-1}(\hat W-W(\la))P)\\&=\rank(\hat W-W(\la^-))-\rank(\hat W-W(\la)):=\rho_{\hat W}(W(\la)).
\end{align*}
In particular, putting $R:=R_{k+1},\,P:=R_k,\,W(\la):=\tS_k(\la),\,\hat W:=\tS_k(\beta)$ we derive \eqref{prop2}.
\par(ii) For the proof of
 \eqref{conbigWa} or \eqref{conbigWb} we put in \eqref{conbig} $\hat W:=W(a)$
or $\hat W:=W(b)$ and use that $\mu(\langle W(\la)\rangle,\langle
W(\la)\rangle=0$ for $\la=a$ or $\la=b.$
\par(iii) Subtracting
\eqref{conbigWa} from \eqref{conbigWb} we derive
\begin{equation*}
    \sum\limits_{a<\nu\le b}\rho_{W(b)}(W(\nu))-\sum\limits_{a<\nu\le b}\rho_{W(a)}(W(\nu))=\mu(\langle W(a)\rangle,\langle W(b)\rangle)+\mu(\langle W(b)\rangle,\langle W(a)\rangle).
\end{equation*}
Then we apply  the property of the comparative index $\mu(Y,\hY)+\mu(\hY,Y)=\rank w(Y,\hY)$ (see \cite[p.448]{E-DE-2009}), where the rank of the Wronskian of $\langle \hS\rangle,\,\langle \cS\rangle$ is equal to  $\rank (\hS-\cS)$ according to \cite[Lemma 2.3 (i)]{E-DE-2010}. The prove of property (iii) is completed.
\end{proof}

Introduce the following notation
\begin{align}\label{Ld}
  L_d(\hZ^{-1}Y,0,N+1):&=\sum\limits_{k=0}^N \mu(\langle  \hZ_{k+1}^{-1}Z_{k+1}
\rangle,\langle  \hZ_{k}^{-1} Z_{k}\rangle),\\\label{Ld*}
L_d^*(Z^{-1}\hY,0,N+1):&=\sum\limits_{k=0}^N \mu(\langle  \hZ_{k}^{-1}Z_{k}
\rangle,\langle  \hZ_{k+1}^{-1} Z_{k+1}\rangle)
\end{align}
for the sums of the terms  $\mu(\langle  \hZ_{k+1}^{-1}Z_{k+1}
\rangle,\langle  \hZ_{k}^{-1} Z_{k}\rangle)$ and $\mu(\langle  \hZ_{k}^{-1}Z_{k}
\rangle,\langle  \hZ_{k+1}^{-1} Z_{k+1}\rangle)$  in \eqref{relnumb}. By analogy with the proof of Proposition~\ref{newosc}(iii) we derive the connection
\begin{equation}\label{Ld+Ld*}
  L_d(\hZ^{-1}Y,0,N+1)+L_d^*(Z^{-1}\hY,0,N+1)=\sum\limits_{k=0}^N\rank(\hS_k-\cS_k),
\end{equation}
where we have used that
\begin{equation*}\begin{aligned}
  \mu(\langle  \hZ_{k}^{-1}Z_{k}
\rangle,\langle  \hZ_{k+1}^{-1} Z_{k+1}\rangle)&+\mu(\langle  \hZ_{k+1}^{-1}Z_{k+1}
\rangle,\langle  \hZ_{k}^{-1} Z_{k}\rangle)\\&=\rank (\Delta (\hZ_{k}^{-1}Z_{k}))=\rank(\hS_k-\cS_k).\end{aligned}
\end{equation*}

The main result of this section is the following theorem which generalizes \cite[Theorem 6.4]{bookDEH} to the case when \eqref{rBconst} does not hold.
\begin{theorem}\label{globoscbig}
Assume  \eqref{dnlaosc:E:Intro.4}
and \eqref{Smon}, then for any $a\in \bR$ there exists the limit $\rank(\cS_k(a)-\cS_k(\la^-))$ and for the numbers $\rho_{\cS_k(a)}(\cS_k(\la))$ defined by
\begin{equation}\label{SaS}
  \rho_k(\la):=\rho_{\cS_{k}(a)}(\cS_k(\la))=\rank(\cS_k(a)-\cS_k(\la^-))-\rank(\cS_k(a)-\cS_k(\la))
\end{equation}
 we have
\begin{equation}\label{renormeigbiga*}
\begin{aligned}
    L_d^*((Z^{[0]}(a))^{-1}Y^{[N+1]}(b),0,N+1)+\sum\limits_{a<\nu\le b}\sum\limits_{k=0}^{N} \rho_k(\nu)=\#\{\nu\in\sigma|\,a<\nu\le b\},
\end{aligned}  \end{equation}
where
\begin{equation}
  \label{L=L*}
 L_d^*((Z^{[0]}(a))^{-1}Y^{[N+1]}(b),0,N+1)=L_d((Z^{[N+1]}(a))^{-1}Y^{[0]}(b),0,N+1)
\end{equation}
and $L_d(\cdot),\,L_d^*(\cdot)$ are given by \eqref{Ld}, \eqref{Ld*}.

Moreover, if \eqref{finrank} holds, then the sum $\sum\limits_{a<\nu\le b}\sum\limits_{k=0}^{N} \rho_k(\nu)$ is finite as $a\rightarrow -\infty$ and for $a<\la_1<\min\sigma$ formula \eqref{renormeigbiga*} presents the number $\#\{\nu\in\sigma|\,\nu\le b\}$ of finite eigenvalues of \eqref{Sla},\eqref{E0} less than or equal to $b.$
\end{theorem}
\begin{proof}
Putting $W(\la):=\cS_{k}(\la)$ in \eqref{conbigWa} and using the notation
$\vartheta(\cS_k(\la))=\vartheta_{k}(\la)$ given by \eqref{jump} we see
that
\begin{equation}\label{conoscbig}
  \sum\limits_{a<\nu\le b}\rho_{\cS_{k}(a)}(\cS_k(\nu))+\mu(\langle \cS_{k}(b)\rangle,\langle \cS_{k}(a)\rangle)=\sum\limits_{a<\nu\le b}\vartheta_{k}(\nu),
\end{equation}
where  $\rho_{\cS_{k}(a)}(\cS_k(\la))$ is given by \eqref{SaS}. Summing \eqref{conoscbig} for $k=0,1,\dots,N$ and then substituting the representation for $\sum\limits_{k=0}^N\sum\limits_{a<\nu\le b}\vartheta_{k}(\nu)=\sum\limits_{a<\nu\le b}\sum\limits_{k=0}^N\vartheta_{k}(\nu)$ into the left-hand side of \eqref{rel1} we cancel the same addends $\mu(\langle \cS_{k}(b)\rangle,\langle \cS_{k}(a)\rangle)$ in the first representation \eqref{relnumb} of the relative oscillation numbers and in \eqref{conoscbig}. Incorporating notation \eqref{Ld*} we derive identity \eqref{renormeigbiga*}.

For the proof of \eqref{L=L*} we replace the roles of $a$ and $b$ in \eqref{princ} and derive
 $$ l_d(Y^{[0]}(a),0,N+1)-l_d(Y^{[0]}(b),0,N+1)=\#(Z^{[N+1]}(a),Z^{[0]}(b),0,N).$$
So we see that $-\#(Z^{[N+1]}(a),Z^{[0]}(b),0,N)=\#(Z^{[N+1]}(b),Z^{[0]}(a),0,N),$ then using for the  relative oscillation numbers in the previous identity the representations associated with  $\mu(\langle \cS_{k}(b)\rangle,\langle \cS_{k}(a)\rangle),\,k=0,1,\dots,N$ (see \eqref{relnumb}) we cancel these terms and derive \eqref{L=L*}.

Using estimate \eqref{prop4} for the comparative index $\mu(\langle \cS_{k}(b)\rangle,\langle \cS_{k}(a)\rangle)\le\rank(\cS_{k}(b)-\cS_{k}(a))\le 2n$ in \eqref{conoscbig} we see that the sum $\sum\limits_{a<\nu\le b}\rho_{\cS_{k}(a)}(\cS_k(\nu))$ is finite if and only if $\sum\limits_{a<\nu\le b}\vartheta_{k}(\nu)$ is finite as $a\rightarrow -\infty.$ Then \eqref{finrank} is sufficient for the finiteness of $\sum\limits_{a<\nu\le b}\sum\limits_{k=0}^N\rho_{\cS_{k}(a)}(\cS_k(\nu)).$ Finally, recall that condition \eqref{finrank} is also sufficient for \eqref{finspectr}. The proof is completed.
\end{proof}
\begin{remark}\label{aboutcon}
\par(i) One can verify directly that \eqref{renormeiga*}, \eqref{renormeiga} in Theorems~\ref{globoscrel1*},~\ref{globoscrel1} are equivalent to  \eqref{renormeigbiga*}, \eqref{L=L*} using the representations for \eqref{Ld*} and \eqref{Ld} in form
\begin{align}\label{Ld*1}
  L_d^*(Z^{-1}\hY,0,N+1)&=l_d^*(Z^{-1}\hY,0,N+1)+\sum\limits_{k=0}^N \mu(\langle  \tS_k
\rangle,\langle  I\rangle),\,\tS_k=Z_{k+1}^{-1}\hS_k Z_{k},\\\label{Ld1} L_d(\hZ^{-1}Y,0,N+1)&=l_d(\hZ^{-1}Y,0,N+1)+\sum\limits_{k=0}^N \mu(\langle  \tS_k
\rangle,\langle  I\rangle),\,\tS_k=\hZ_{k+1}^{-1}\cS_k \hZ_{k}
\end{align}
which is based on \cite[Lemma 2.3(v)]{E-DE-2010} and  \eqref{backwind}, \eqref{forwind} (see also  \cite[Remark 4.46]{bookDEH}). Next one can apply \eqref{conbigWa} putting   $W(\la):=\tS_k(\la)$ to see that \eqref{renormeigbiga*} is equivalent to \eqref{renormeiga*} and that \eqref{renormeigbiga*} with  $L_d^*((Z^{[0]}(a))^{-1}Y^{[N+1]}(b),0,N+1)$ replaced by  $L_d((Z^{[N+1]}(a))^{-1}Y^{[0]}(b),0,N+1)$ is equivalent to \eqref{renormeiga}.

\par(ii) Using the relation $\sum\limits_{a<\nu\le b}(\rho_{S_k(b)}(S_k(\nu))-\rho_{S_k(a)}(S_k(\nu)))=\rank(\cS_k(b)-\cS_k(a))$ according to Proposition~\ref{newosc}(iii) and incorporating connection \eqref{Ld+Ld*} one can derive the equivalent form of \eqref{renormeigb}, \eqref{renormeigb*} in Theorems~\ref{globoscrel1},~\ref{globoscrel1*}
\begin{equation*}\label{renormeigbigb}
    -L_d((Z^{[N+1]}(b))^{-1}Y^{[0]}(a),0,N+1)+\sum\limits_{a<\nu\le b} \sum\limits_{k=0}^{N} \rho_{\cS_k(b)}(\cS_k(\nu))=\#\{\nu\in\sigma|\,a<\nu\le b\},
\end{equation*}
where
\begin{align}\label{L*=L}
 L_d((Z^{[N+1]}(b))^{-1}Y^{[0]}(a),0,N+1)&= L_d^*((Z^{[0]}(b))^{-1}Y^{[N+1]}(a),0,N+1).
\end{align}
\par(iii) The main advantage of Theorem~\ref{globoscbig} is  the invariant form of the sum $\sum\limits_{a<\nu\le b}\sum\limits_{k=0}^{N} \rho_k(\nu)$ which does not depend on the (unknown) transformation matrices $Z_k^{[l]}(\beta),\,l\in\{0,N+1\},\,\beta\in\{a,b\}$ as it takes place in Theorems~\ref{globoscrel1},~\ref{globoscrel1*}. The price of this advantage is the necessity to calculate $L_d^*(Z^{-1}\hY,0,N+1)$ or $L_d(\hZ^{-1}Y,0,N+1)$ instead of $l_d^*(Z^{-1}\hY,0,N+1)$ or $l_d(\hZ^{-1}Y,0,N+1)$ according to connections \eqref{Ld*1}, \eqref{Ld1}. Remark also that $L_d^*(Z^{-1}\hY,0,N+1)$ and $L_d(\hZ^{-1}Y,0,N+1)$ have the same meaning as $l_d^*(Z^{-1}\hY,0,N+1)$ and  $l_d(\hZ^{-1}Y,0,N+1)$, i.e., present the number of backward and forward focal points of conjoined bases of some augmented systems associated with \eqref{tSla} (see  \cite[Remark 4.46]{bookDEH} for more details).
\end{remark}
As a corollary to Theorems~\ref{globoscrel1},~\ref{globoscrel1*},~\ref{globoscbig} consider  the important special case associated with the condition
\begin{equation}\label{nonoscbiG}\begin{array}{c}
  \rho_k(\la):=\rho_{\cS_{k}(a)}(\cS_k(\la))=\rank(\cS_k(a)-\cS_k(\la^-))-\rank(\cS_k(a)-\cS_k(\la))=0,\\\la\in(a,b],\,k\in[0,N]_{\sZ}.
\end{array}\end{equation}

\begin{theorem}\label{relth1}
Assume  \eqref{dnlaosc:E:Intro.4},
 \eqref{Smon}. Then, for any $a<b$  condition \eqref{nonoscbiG} is necessary and sufficient for the representation
 \begin{equation}\label{renormeigconB}
 \begin{aligned}
  L_d^*((Z^{[0]}(a))^{-1}Y^{[N+1]}(b),0,N+1)&=L_d((Z^{[N+1]}(a))^{-1}Y^{[0]}(b),0,N+1)\\&=\#\{\nu\in\sigma |a <\nu\le b\}.
  \end{aligned}
\end{equation}

Similarly, condition \eqref{nonoscbiG} is necessary and sufficient  for the representations of
the sums
\begin{equation}\label{ad10conB}
 \sum\limits_{a<\nu\le b}n_{Z^{[l]}(a)}(\nu)= \sum\limits_{k=0}^{N} \mu(\langle \tS_k(b)\rangle,\langle I \rangle)\le n(N+1),\,l\in\{0,N+1\}
\end{equation}
in \eqref{renormeiga} and \eqref{renormeiga*}, where $\tS_k(b)$ is the coefficient matrix of \eqref{tSla} associated with $R_k:=Z^{[l]}(a),\,l\in\{0,N+1\}.$
\end{theorem}
\begin{proof}
By \eqref{renormeigbiga*}, \eqref{L=L*} we see that \eqref{nonoscbiG} is equivalent to \eqref{renormeigconB}, where we use that numbers \eqref{SaS} are nonnegative.

 Putting in \eqref{conbigWa} of Proposition~\ref{newosc}(ii) $W(\la):=\tS_k(\la),$ where $\tS_k(\la)=(Z_{k+1}^{[l]}(a))^{-1}\cS_k(\la)Z_{k}^{[l]}(a),\,l\in\{0,N+1\}$ and incorporating that $\vartheta(\tS_k(\la))=\tilde \vartheta_k(\la)$  for $\tilde \vartheta_k(\la)$ given by \eqref{jumpt} we have
\begin{equation*}
  \sum\limits_{a<\nu\le b}\rho_{\tS_k(a)}(\tS_k(\nu))-\sum\limits_{a<\nu\le b}\tilde \vartheta(\nu)=-\mu(\langle \tS_k(b)\rangle,\langle \tS_k(a)\rangle).
\end{equation*}
 Observe also that $\tS_k(a)=(Z_{k+1}^{[l]}(a))^{-1}\cS_k(a)Z_{k}^{[l]}(a)=I,\,l\in\{0,N+1\}$ then summing the above identity from $k=0$ to $k=N$ and incorporating property \eqref{prop2} we derive for $l\in\{0,N+1\}$
\begin{equation}\label{ad10}
 \sum\limits_{a<\nu\le b}n_{Z^{[l]}(a)}(\nu)= \sum\limits_{a<\nu\le b} \sum\limits_{k=0}^{N} \rho_{S_k(a)}(S_k(\nu))+\sum\limits_{k=0}^{N} \mu(\langle \tS_k(b)\rangle,\langle I \rangle).
\end{equation}
By \eqref{ad10} we see that \eqref{nonoscbiG} is equivalent to \eqref{ad10conB}, where we estimate the comparative index $\mu(\langle \tS_k(b)\rangle,\langle I \rangle)\le n$ using \eqref{prop4}. The proof is completed.
\end{proof}

Next we formulate the  simplest sufficient criteria for \eqref{nonoscbiG}.
\begin{corollary}\label{underconB}
Assume \eqref{dnlaosc:E:Intro.4},
 \eqref{Smon} and
\begin{equation}\label{conB}
  \rank\cB_k(\la^-)=\rank\cB_k(\la),\,\la\in (a,b],\,k\in[0,N]_{\sZ}
\end{equation}
for the block $\cB_k(\la)$  of $\cS_k(\la)$ given by
\eqref{dnlaosc:E:Intro.4}. Then condition \eqref{nonoscbiG} and
the majorant condition \eqref{maj} hold. Moreover, according to
Theorem~\ref{relth1} we also have  \eqref{renormeigconB} and \eqref{ad10conB}.
\end{corollary}
\begin{proof}
For the proof we use \eqref{conoscbig}. Applying \eqref{conB} we see that $\sum\limits_{a<\nu\le
b}\vartheta_{k}(\nu)=0.$ Remark that both addends in the left-hand
side of \eqref{conoscbig} are nonnegative, then it follows that
\begin{equation}\label{addd}
    \mu(\langle \cS_{k}(b)\rangle,\langle \cS_{k}(a)\rangle)=0,\quad \rho_{\cS_{k}(a)}(\cS_k(\la))=0,\,k\in [0,N],\;\la\in (a,b].
\end{equation}
Then we have proved \eqref{nonoscbiG} and \eqref{maj}. Moreover, by \eqref{nonoscbiG} we have that \eqref{renormeigconB} and \eqref{ad10conB} hold.
\end{proof}
\begin{remark}\label{omitted ex}
\par(i) Consider  problem \eqref{Sla},\eqref{E0} under \eqref{Smon} and  the additional assumption for the matrix $\cS_k(\la)$
\begin{equation*}\label{constupp}
  \cA_k(\la)\equiv\cA_k,\quad \cB_k(\la)\equiv\cB_k,\quad k\in[0,N]_{\sZ},\;\la\in\bR
\end{equation*}
which covers the special case \eqref{dnlaosc:E:Intro.7A}. We see that condition \eqref{conB} is satisfied for all $a<b,$ then by Corollary~\ref{underconB} all identities in Theorem~\ref{relth1} hold. Moreover, one can verify directly (see also \cite[Subsection 6.1.3]{bookDEH}) that the block $\tB_k(\la)$ of the matrix $\tS_k(\la)$ associated with $R_k=Z_k^{[l]}(a),\,l\in\{0,N+1\}$ is symmetric, $\tB_k(a)=0,$ and by \eqref{Smon} we have $\dot{\tB}_k(\la)\ge 0.$ Finally, applying Corollary~\ref{kratzcor} we derive for the left-hand side of \eqref{ad10conB}
$ \sum\limits_{a<\nu\le b}n_{Z^{[l]}(a)}(\nu)=0,\,l\in\{0,N+1\},$
then, instead of \eqref{renormeigconB} we have by Theorems~\ref{globoscrel1},~\ref{globoscrel1*} (see also \eqref{Ld*1}, \eqref{Ld1})
\begin{equation*}\label{backwforwlow}
  l_d^*((Z^{[0]}(a))^{-1}Y^{[N+1]}(b),0,N+1)=l_d((Z^{[N+1]}(a))^{-1}Y^{[0]}(b),0,N+1)=\#\{\nu\in\sigma|a<\nu\le b\},
\end{equation*}
i.e., the number of finite eigenvalues of \eqref{Sla},\eqref{E0} in $(a,b]$ can be calculated using the number of backward focal points in $[0,N+1)$ of the conjoined basis $(Z^{[0]}(a))^{-1}Y^{[N+1]}(b)$ which equals the number of forward focal points in $(0,N+1]$ of the conjoined basis $(Z^{[N+1]}(a))^{-1}Y^{[0]}(b).$
This result was proved in \cite[Theorem 6.9 and Remark 6.10(i)]{bookDEH}, where we used a different proof.
\par(ii) The most important special case of \eqref{Sla} for which condition \eqref{conB} is satisfied is  discrete matrix Sturm-Liouville eigenvalue problems with nonlinear dependence on $\la\in\bR$. Renormalized and more general relative oscillation theorems for these problems are presented in \cite{E-ADV} (see also \cite[Section 6.1]{bookDEH}).
\par(iii) Observe that under the assumptions of Corollary~\ref{underconB} formula \eqref{renormeigconB} is equivalent to \begin{equation}\label{rel11}
    \#(Z^{[N+1]}(b),Z^{[0]}(a),0,N)=\#\{\nu\in\sigma |a <\nu\le b\},
  \end{equation}
  where for the relative oscillation numbers $\#(Z^{[N+1]}(b),Z^{[0]}(a),0,N)$   given by  \eqref{globrelnu} with $\hZ:=Z^{[N+1]}(b),$
  $Z:=Z^{[0]}(a)$ we have majorant condition \eqref{maj} and then
  \begin{equation*}\label{globrelnumaj}
  \#_k(Z^{[N+1]}(b),Z^{[0]}(a))=\mu(\langle (Z^{[N+1]}_k(b))^{-1}Z^{[0]}_k(a) \rangle,\langle Z^{[N+1]}_{k+1}(b))^{-1}Z^{[0]}_{k+1}(a)\rangle)\ge 0,
\end{equation*}
i.e.,  \eqref{relnon} holds. This result was derived  in \cite[Theorem 6.4 and Remark 6.5(i)]{bookDEH}, where we used a different proof.
\end{remark}
Combining the proof of Corollary~\ref{underconB} and Proposition~\ref{newosc}(i) one can generalize Corollary~\ref{underconB} as follows.
\begin{corollary}\label{underconBgen}
Assume \eqref{dnlaosc:E:Intro.4} and \eqref{Smon} and suppose that there exist
symplectic matrices $R$ and $P$ such that for the matrix
\begin{equation}\label{tS}
   \bar \cS_k(\la)= R^{-1}\cS_{k}(\la)P=\smat{\bar A_{k}(\la)}{\bar B_{k}(\la)}{\bar C_{k}(\la)}{\bar D_{k}(\la)}
\end{equation}
condition \eqref{conB} is satisfied, i.e.,
\begin{equation}\label{contB}
  \rank\bar B_k(\la^-)=\rank\bar B_k(\la),\,\la\in (a,b],\,k\in[0,N]_{\sZ}.
\end{equation}
Then we have \eqref{nonoscbiG}
and  according to
Theorem~\ref{relth1} identities \eqref{renormeigconB} and \eqref{ad10conB} hold.
\end{corollary}
\begin{proof}
As it was proved in Corollary~\ref{underconB} condition \eqref{contB} is
sufficient for $\rho_{\bar \cS_k(a)}(\bar \cS_k(\la))=0$, then by
Proposition~\ref{newosc}(i) we prove that $\rho_{\cS_k(a)}( \cS_k(\la))=0.$
\end{proof}
In particular, criterion \eqref{contB} is satisfied if one of the blocks $\cA_k(\la),\,\cC_k(\la),\,\cD_k(\la)$ in representation \eqref{dnlaosc:E:Intro.4} of $\cS_k(\la)$ is nonsingular in $(a,b].$ This assumption is true for the block $\cA_k(\la)$  of $\cS_k(\la)$ associated with the most important special case of \eqref{Sla}, with the discrete
Hamiltonian systems \cite{cdA.acP96}. Remark that condition \eqref{conB} is not assumed.
\begin{example}\label{Hamex}
Consider the discrete Hamiltonian eigenvalue problem
    \begin{gather}
\label{LHS} \Delta x\K=A_k(\lambda)x\K+B_k(\lambda)u_k,\quad \Delta
u_k=C_k(\lambda) x\K-A^T_k(\lambda)u_k ,\\\notag \det \left( {I -
{A}_{k}(\lambda)} \right) \ne 0,\, \,k=0,\dots,N,\\ \label{dirbh}
         x_0(\lambda)=x_{N+1}(\lambda)=0,
 \end{gather}
with the Hamiltonian
$$\cH_{k}(\lambda)=\cH_{k}^{T}(\lambda),\quad \cH_{k}(\lambda)=\begin{pmatrix}
                                           - {C}_{k}(\lambda) &  {A}_{k}^{T}(\lambda) \\
                                             {A}_{k}(\lambda) &   {B}_{k}(\lambda) \\

                                         \end{pmatrix}$$
which obeys the monotonicity condition (see \cite[Example 7.9]{rSH12})
\begin{equation}\label{monham}
    \dot{\cH}_{k}(\lambda)\ge 0,\,\lambda\in \mathbb{R}.
\end{equation}
For the Hamiltonian system \eqref{LHS} rewritten in form
\eqref{Sla} the matrix $\cS_{k}(\la)$ is given by
\begin{equation}                                      \label{LHS-SDS}
\cS_{k}(\lambda)=\begin{pmatrix} (I-A_k(\lambda))^{-1} & \quad(I-A_k(\lambda))^{-1}B_k(\lambda)\\
C_k(\lambda)(I-A_k(\lambda))^{-1} & \quad C_k(\lambda)(I-A_k(\lambda))^{-1}B_k(\lambda) +I-A^T_k(\lambda)\end{pmatrix}.
\end{equation}
The block $\cA_{k}(\la)=(I-A_k(\lambda))^{-1}$ of $\cS_{k}(\la)$ is
nonsingular, then there exists the matrices $R=I$ and $P=J$ such
that $\bar \cS_{k}(\la):=\cS_{k}(\la)J$ has the nonsingular block
$\tB_{k}(\la),$ i.e., condition \eqref{contB} is satisfied.
Applying Corollary~\ref{underconBgen} we derive that $\cS_{k}(\la)$ given
by \eqref{LHS-SDS} obeys condition \eqref{nonoscbiG}.
\begin{corollary}\label{renham}
Consider problem \eqref{LHS}, \eqref{dirbh} under assumption \eqref{monham}, then for any $a<b$ we have condition \eqref{nonoscbiG} for
matrix \eqref{LHS-SDS} and
according to
Theorem~\ref{relth1} identity \eqref{renormeigconB}  holds for the number $\#\{\nu\in\sigma |a <\nu\le b\}$ of finite eigenvalues of \eqref{LHS}, \eqref{dirbh}.
\end{corollary}
Recall that condition \eqref{conB} for
$\cS_{k}(\la)$ is not assumed, moreover, applying
Corollary~\ref{kratzcor} for
$Q(\la):=B_{k}(\la)=\cA_{k}^{-1}(\la)\cB_{k}(\la)$ we see that
\begin{equation}\label{Hamfoc}
\sum\limits_{a<\nu\le b}\vartheta_{k}(\nu)=\ind B_{k}(a)-\ind B_{k}(b)=\mu(\langle \cS_{k}(b)\rangle,\langle\cS_{k}(a)\rangle),
\end{equation}
where we have used that $\rank \cB_{k}(\la)=\rank B_{k}(\la)$,
\eqref{rankQind} and \eqref{conoscbig}.
\end{example}
\begin{remark}
The results of this paper can be further  developed to the case of two symplectic eigenvalue problems in form of \eqref{Sla} with different coefficient matrices $\cS_k(\la),\,\hS_k(\la)$, which are the subject of the relative oscillation theory. This theory is developed in \cite[Chapter 6]{bookDEH} for $\cS_k(\la),\,\hS_k(\la)$ under restriction \eqref{rBconst}.  Using new comparison results in \cite{jE20} we are going to generalise the relative oscillation theory to the case when this condition is omitted.
\end{remark}
\section*{Acknowledgements}   This research is supported by Federal
Programme of Ministry of Education and Science of the Russian
Federation in the framework of the state order [grant number
2014/105] and the Czech Science Foundation under grant
GA19--01246S.

\appendix

\section{Proof of Lemma~\ref{newosc1}}\label{sec5}
Introduce the $4n\times 4n$ matrices
\begin{equation}\label{R}\begin{array}{c}
\hat {\cZ}(\la)=R^{-1}\{I,W(\la)\hat W^{-1}\}R,\quad
\cZ(\la)=R^{-1}\{I,W(\la)\},\\[2mm] \{I,W(\la)\}=\begin{pmatrix}
                   I & 0 & 0 & 0 \\
                   0 & A(\la) & 0 & B(\la) \\
                   0 & 0 & I & 0 \\
                   0 & C(\la) & 0 & D(\la) \\
                 \end{pmatrix},\;  R=\frac{1}{\sqrt{2}}\begin{pmatrix}
                                                                                0 & -I & I & 0 \\
                                                                                0 & I & I & 0 \\
                                                                                -I & 0 & 0 & -I \\
                                                                                -I & 0 & 0 & I \\
                                                                              \end{pmatrix}.
                                                                              \end{array}\end{equation}
Then it easy to verify that $R$ is symplectic (and orthogonal) and $\{I,W(\la)\},\,\cZ(\la),\,\hat {\cZ}(\la)\in \bR^{4n\times 4n}$ are symplectic provided $W(\la),\,\hat W\in\Sp$, moreover $$\Psi(\cZ(\la))=R^{T}\Psi(\{I,W(\lambda)\})R=R^{T}\{0,\Psi(W(\lambda))\}R\ge 0$$ and
\begin{align*}
  \Psi(\hat {\cZ}(\la))&=R^{T}\Psi(\{I,W(\lambda)\hat W^{-1}\})R
    =R^{T}\{0,\Psi(W(\lambda)\hat W^{-1})\}R=R^{T}\{0,\Psi(W(\lambda))\}R\ge 0
\end{align*}
provided \eqref{psimon} holds. Observe that the assumptions of Theorem~\ref{contsep} are satisfied because $\hat{\cZ}^{-1}(\la)Z(\la)=R^{-1}\{I,\hat W\}$ does not depend on $\la$. Applying \eqref{elyscond51sep}  we have
\begin{equation}\label{ad5}
\begin{aligned}
  l_c(\hat{\cZ}(\la)\bnn,a,b)&-l_c({\cZ}(\la)\bnn,a,b)=\mu({\cZ}(\la)\bnn,\hat{\cZ}(\la)\bnn)|_b^a.
\end{aligned}\end{equation}
Using Property 3 of the comparative index (see \cite[p. 448]{E-DE-2009} or \cite[Theorem~3.5~(iii)]{bookDEH}) and then \cite[Lemma~2.3(v)]{E-DE-2010} (see also \cite[Lemma~3.21~(iii)]{bookDEH}) we simplify the right-hand side of \eqref{ad5} according to
\begin{align*}
\mu({\cZ}(\la)\bnn,\hat{\cZ}(\la)\bnn)&=\mu^*({\cZ}^{-1}(\la)\bnn,{\cZ}^{-1}(\la)\hat{\cZ}(\la)\bnn)\\&=\mu^*(\langle W(\la)^{-1} \rangle,\langle \hat W^{-1}\rangle)=\mu(\langle W(\la) \rangle,\langle \hat W\rangle).
\end{align*}
Moreover, the upper blocks of the matrices $\hat{\cZ}(\la)\bnn,\,{\cZ}(\la)\bnn$ have the form
\begin{equation*}
 (I\;0) \hat{\cZ}(\la)\bnn=\frac{1}{2}\cJ(I-W(\la)\hat W^{-1}),\;(I\;0) {\cZ}(\la)\bnn=\frac{1}{\sqrt{2}}\smat{-I}{-A^T(\la)}{0}{-B^T(\la)},
\end{equation*}
 where the rank of the second matrix in the above formula  is equal to $n+\rank B(\la).$ Then
\begin{equation*}
l_c(\hat{\cZ}(\la)\bnn,a,b)=\sum\limits_{a<\nu\le b}\rho_{\hat W}(W(\nu)),\;\;l_c({\cZ}(\la)\bnn,a,b)=\sum\limits_{a<\nu\le b}\vartheta (W(\nu)).
\end{equation*}

Substituting the representations derived above into \eqref{ad5} we prove formula \eqref{conbig}.


\begin{thebibliography}{32}

\bibitem[1]{cdA.acP96}
C.~D.~Ahlbrandt, A.~C.~Peterson,
{\em Discrete Hamiltonian Systems. Difference Equations, Continued Fractions, and Riccati Equations}
Kluwer Texts in the Mathematical Sciences, Vol.~16, Kluwer Academic Publishers Group, Dordrecht, 1996.

\bibitem{teschl1}
 K. Ammann and G. Teschl, {\em Relative
    Oscillation Theory for Jacobi Matrices}, In: Proceedings of
    the 14th International Conference on Difference Equations
    and Applications, M. Bohner (ed) et al.
    U\u{g}ur--Bah\c{c}e\c{s}ehir University Publishing Company,
    Istanbul (2009), pp. 105--115.
\bibitem{amman}
K.~Ammann,
{\em Relative oscillation theory for Jacobi matrices extended},
{Oper. Matrices} {\bf 1} (2014), 99--115.

\bibitem{mB98b}
M.~Bohner,
{\em Discrete linear Hamiltonian eigenvalue problems},
{Comput. Math. Appl.} {\bf 36} (1998), no.~10--12, 179--192.

\bibitem{mB.oD97}
M.~Bohner, O.~Do\v{s}l\'y,
{\em Disconjugacy and transformations for symplectic systems},
{Rocky Mountain J. Math.} { 27} (1997), no.~3, 707--743.

\bibitem{B-D-K-RMJM}
M.~Bohner, O.~Do\v{s}l\'y, W.~Kratz.
{\em An oscillation theorem for discrete eigenvalue problems},
{Rocky Mountain J. Math.} { 33} (2003), no.~4, 1233--1260.

\bibitem{BHK} M. Bohner, W. Kratz, R.\v{S}imon Hilscher, {\em Oscillation and spectral theory for linear Hamiltonian systems
with nonlinear dependence on the spectral parameter}, Math. Nachr. 285, No.11--12, (2012), 1343--1356.


\bibitem{D-K-JDEA-2007}
O. Do\v{s}l\'y, W. Kratz,
{\em Oscillation theorems for symplectic difference systems},
{J. Difference Equ. Appl.} { 13} (2007), 585--60.

\bibitem{D-Kyoto}
{O. Do\v sl\'y},
    {\em Oscillation theory of symplectic difference systems},
    Adv. Stud. Pure Math., 53,
    Math. Soc. Japan, Tokyo, 2009.

\bibitem{D-K-JDEA-2010}
O. Do\v{s}l\'y, W. Kratz,
{\em Oscillation and spectral theory for symplectic difference systems with separated boundary conditions},
{J. Difference Equ. Appl.} {\bf 16} (2010), 831--846.

\bibitem{D-E-JDEA}
O. Do\v{s}l\'y, J. Elyseeva,
{\em Singular comparison theorems for discrete symplectic systems},
{J. Difference Equ. Appl.} {20} (2014), no.~8, 1268--1288.


\bibitem{bookDEH} O.~Do\v{s}l\'y, J. Elyseeva,
    R.~\v{S}imon~Hilscher, {\em Symplectic difference systems:
    oscillation and spectral theory}, Birkh\"{a}user Basel, DOI: 10.1007/978-3-030-19373-7, ISBN: 978-3-030-19373-7.

\bibitem{E-DE-2009}
Yu. V. Eliseeva,
{\em Comparative index for solutions of symplectic difference systems},
{Differential Equations} { 45} (2009), no.~3, 445--459.


\bibitem{E-DE-2010}
Yu. V. Eliseeva,
{\em Comparison theorems for symplectic systems of difference equations},
{Differential Equations} { 46} (2010), no.~9, 1339--1352.

\bibitem{E-AML2010}
 J. Elyseeva, {\em On relative oscillation theory for symplectic eigenvalue
 problems}
, Appl. Math. Letters, 23 (2010), pp. 1231--1237.

\bibitem{E-AML2012}
J. V. Elyseeva,
{\em A note on relative oscillation theory for symplectic difference systems with general boundary conditions},
{Appl. Math. Lett.} { 25} (2012), no.~11, 1809--1814.

\bibitem{E-ADV}
 {J.V. Elyseeva},
    {\em Relative oscillation theory for matrix
Sturm-Liouville difference equations extended},
    Adv. Differ. Equ. 2013,  (2013:328)
Available at
http://www.advancesindifferenceequations.com/content/2013/1/328

\bibitem{jE15}
J.~V. Elyseeva,
{\em Generalized oscillation theorems for symplectic difference systems with nonlinear dependence on spectral parameter},
{Appl. Math. Comput.} { 251} (2015), 92--107.

\bibitem{jE16}
J.~V.~Elyseeva,
{\em Comparison theorems for conjoined bases of linear Hamiltonian differential systems and the comparative index},
{J. Math. Anal. Appl.} { 444} (2016), no.~2, 1260--1273.


\bibitem{jvE18}
J.~V.~Elyseeva,
{\em The comparative index and transformations of linear Hamiltonian differential systems},
{Appl. Math. Comput.} { 330} (2018), 185--200.


\bibitem{jvE.rSH18}
J.~V.~Elyseeva, R.~\v{S}imon~Hilscher,
{\em Discrete oscillation theorems for symplectic eigenvalue problems with general boundary conditions depending nonlinearly on spectral parameter},
{Linear Algebra Appl.} { 558} (2018), 108--145.

\bibitem{jE20} J. Elyseeva, {\em Comparison theorems for conjoined bases of linear
Hamiltonian systems without monotonicity}, {Monatsh. Math.} (2020),
DOI 10.1007/s00605-020-01378-8

\bibitem{Gesztesy1} F. Gesztesy, B. Simon, G. Teschl, {\em Zeros of the Wronskian and renormalized oscillation theory}, American J. Math. 118 (1996) 571--594.

\bibitem{Gesztesy2} F. Gesztesy, M. Zinchenko, {\em Renormalized oscillation theory for Hamiltonian systems}, Adv. Math. 311 (2017) 569–597.

\bibitem{Howard} P. Howard, A. Sukhtayev, {\em Renormalized Oscillation Theory for Linear Hamiltonian
Systems on [0, 1] via the Maslov Index}, 2018,	arXiv:1808.08264 [math.CA].

\bibitem{K-JDEA}
 {W. Kratz},
    {\em Discrete oscillation}, {J.
    Difference Equ. Appl.} {9} (2003), pp. 135--147.


\bibitem{Kratz-Analysis}
W.~Kratz,
{\em Definiteness of quadratic functionals},
{Analysis (Munich)} { 23} (2003), 163--183.

\bibitem{wK.rSH13}
W.~Kratz, R.~\v{S}imon~Hilscher,
{\em A generalized index theorem for monotone matrix-valued functions with applications to discrete oscillation theory},
{SIAM J. Matrix Anal. Appl.} { 34} (2013), no.~1, 228¨C-243.

\bibitem{Kruger} H. Kr\"{u}ger, G. Teschl, {\em Relative oscillation theory, weighted zeros of the Wronskian, and the spectral shift function}, Commun. Math. Phys. 287 (2009), 613--640.
\bibitem{pS.rSH17}
P.~\v{S}epitka, R.~\v{S}imon~Hilscher,
{\em Comparative index and Sturmian theory for linear Hamiltonian systems},
{J. Differential Equations} { 262} (2017), no.~2, 914--944.

\bibitem{rSH12}
R.~\v{S}imon~Hilscher,
{\em Oscillation theorems for discrete symplectic systems with nonlinear dependence in spectral parameter},
{Linear Algebra Appl.} { 437} (2012), no.~12, 2922--2960.

\bibitem{Hilsch4} R. \v{S}imon  Hilscher, {\em On General Sturmian theory for abnormal linear Hamiltonian systems}, Discrete and Continuous Dynamical systems  (2011),  684--691.

\bibitem{teshl3}
G. Teschl,
{\em Oscillation theory and renormalized oscillation theory for Jacobi operators},
{J. Differential Equations} { 129} (1996), 532--558.
\end{thebibliography}
\end{document}